\documentclass[twoside,a4paper,nointlimits]{amsart}

\usepackage{amsmath}
\usepackage{amsthm}
\usepackage{amssymb}
\usepackage{amsfonts}
\usepackage{amsxtra}
\usepackage{enumitem}
\usepackage{color}

\usepackage{graphicx}
\usepackage[T1]{fontenc} 
\usepackage{txfonts}
\usepackage{comment}

\def\phi{\varphi}
\def\epsilon{\varepsilon}

  \providecommand{\normtmp}[2]{#1\lVert{#2}#1\rVert}
  \providecommand{\norm}[1]{\normtmp{}{#1}}
  \providecommand{\bignorm}[1]{\normtmp{\big}{#1}}
  \providecommand{\Bignorm}[1]{\normtmp{\Big}{#1}}
  \providecommand{\biggnorm}[1]{\normtmp{\bigg}{#1}}
  \providecommand{\Biggnorm}[1]{\normtmp{\Bigg}{#1}}

  \providecommand{\abstmp}[2]{#1\lvert{#2}#1\rvert}
  \providecommand{\abs}[1]{\abstmp{}{#1}}
  \providecommand{\bigabs}[1]{\abstmp{\big}{#1}}
  
  \providecommand{\biggabs}[1]{\abstmp{\bigg}{#1}}

  \providecommand{\settmp}[2]{#1\{{#2}#1\}}
  \providecommand{\set}[1]{\settmp{}{#1}}
  \providecommand{\bigset}[1]{\settmp{\big}{#1}}
  \providecommand{\Bigset}[1]{\settmp{\Big}{#1}}

  \providecommand{\skptmp}[3]{\ensuremath{#1\langle {#2}, {#3} #1\rangle}}
  \providecommand{\skp}[2]{\skptmp{}{#1}{#2}}

\newcommand{\dashint}{\fint}

\numberwithin{equation}{section}

\theoremstyle{plain}
\newtheorem{theorem}[equation]{Theorem}
\newtheorem{lemma}[equation]{Lemma}
\newtheorem{proposition}[equation]{Proposition}
\newtheorem{corollary}[equation]{Corollary}

\theoremstyle{definition}
\newtheorem{definition}[equation]{Definition}

\newtheorem*{standingassumption}{Standing Assumptions}

\theoremstyle{remark}
\newtheorem{remark}[equation]{Remark}

 \providecommand{\loc}{{\ensuremath{\mathrm{loc}}}}

\newcommand{\Fsp}[2]{{F^{#1}_{#2}}}
\newcommand{\Fspq}[3]{{F^{#1}_{#2,\,#3}}}

\newcommand{\Bspq}[3]{{{B}^{#1}_{#2,\,#3}}}

\newcommand{\fsp}[2]{{f^{#1}_{#2}}}
\newcommand{\fspq}[3]{{f^{#1}_{#2,\,#3}}}

\newcommand{\lqxnu}{{l^{q(x)}_\nu}}
\newcommand{\lqxmu}{{l^{q(x)}_\mu}}
\newcommand{\Lpdotx}{{L^{p(\cdot)}_x}}
\newcommand{\lqrxnu}{{l^{\frac{q(x)}{r}}_\nu}}
\newcommand{\lqrxmu}{{l^{\frac{q(x)}{r}}_\mu}}

\newcommand{\Lprdotx}{{L^{\frac{p(\cdot)}{r}}_x}}

\newcommand{\R}{\mathbb{R}}
\newcommand{\N}{\mathbb{N}}
\newcommand{\Z}{\mathbb{Z}}
\newcommand{\C}{\mathbb{C}}
\newcommand{\Rn}{{\mathbb{R}^n}}
\def\ln{\log}
\renewcommand{\le}{\leqslant}
\renewcommand{\ge}{\geqslant}
\renewcommand{\leq}{\leqslant}
\renewcommand{\geq}{\geqslant}

 \DeclareMathOperator{\supp}{supp}

\newcommand{\cS}{\mathcal{S}}
\newcommand{\ds}{\displaystyle}
\newcommand{\ip}[2]{\left\langle#1,#2\right\rangle}

\def\esssup{\operatornamewithlimits{ess\,sup}}
\def\essinf{\operatornamewithlimits{ess\,inf}}

  \DeclareMathOperator{\trace}{tr}

\begin{document}


\title{Function spaces of variable smoothness and integrability}

\author[L.\ Diening]{Lars Diening$^*$}
\thanks{$^*$ Supported in part by the Landesstiftung
  Baden-W{\"u}rttemberg}
\address{Section of Applied Mathematics, Eckerstra\ss{}e~1,
Freiburg University\\ 79104~Frei\-burg/Breis\-gau, Germany}
\email{diening[]mathematik.uni-freiburg.de}
\urladdr{http://www.mathematik.uni-freiburg.de/IAM/homepages/diening/}

\author[P.\ H\"ast\"o]{Peter H{\"a}st{\"o}$^\dagger$}
\address{Department of Mathematical Sciences,
P.O.~Box~3000\\ FI-90014 University of Oulu, Finland}
\email{peter.hasto[]helsinki.fi}
\urladdr{http://cc.oulu.fi/$\sim$phasto/}
\thanks{$^\dagger$ Supported in part by the Academy of Finland}

\author[S.\ Roudenko]{Svetlana Roudenko$^\ddagger$}
\address{Department of Mathematics, Arizona State University \\
Tempe, AZ~85287-1804, USA}
\email{svetlana[]math.asu.edu}
\urladdr{http://math.asu.edu/$\sim$svetlana}
\thanks{$^\ddagger$ Partially supported by the NSF grant DMS-0531337}

\begin{abstract}
  In this article we introduce Triebel--Lizorkin spaces with variable
  smoothness and integrability. Our new scale covers spaces with
  variable exponent as well as spaces of variable smoothness that have
  been studied in recent years. Vector-valued maximal inequalities do
  not work in the generality which we pursue, and an alternate
  approach is thus developed. Applying it, we give molecular and atomic
  decomposition results and show that our space is well-defined, i.e.,
  independent of the choice of basis functions.

  As in the classical case, a unified scale of spaces permits clearer
  results in cases where smoothness and integrability interact, such
  as Sobolev embedding and trace theorems.  As an application of our
  decomposition we prove optimal trace theorems in the variable
  indices case.
\end{abstract}



\maketitle

\setcounter{tocdepth}{1}


\section{Introduction}
\label{sect:introduction}

{}From a vast array of different function spaces a well ordered
superstructure appeared in the 1960's and 70's based on two
three-index spaces: the Besov space $B^\alpha_{p,q}$ and the
Triebel--Lizorkin space $F^\alpha_{p,q}$.  In recent years there has
been a growing interest in generalizing classical spaces such as
Lebesgue and Sobolev spaces to the case with either variable
integrability (e.g.,\ $W^{1, p(\cdot)}$) or variable smoothness (e.g.,\
$W^{m(\cdot), 2}$). These generalized spaces are obviously not covered
by the superstructures with fixed indices.

It is well-known from the classical case that smoothness and
integrability often interact, for instance, in trace and embedding
theorems. However, there has so far been no attempt to treat spaces
with variable integrability and smoothness in one scale.  In this
article we address this issue by introducing Triebel--Lizorkin spaces
with variable indices, denoted
$\Fspq{\alpha(\cdot)}{p(\cdot)}{q(\cdot)}$.

Spaces of variable integrability can be traced back to 1931 and W.\
Orlicz \cite{Orl31}, but the modern development started with the
paper \cite{KR} of Kov{\'a}{\v c}ik and R{\'a}kosn{\'\i}k in 1991. A
survey of the history of the field with a bibliography of more than
a hundred titles published up to 2004 can be found in \cite{DHN} by
Diening, H{\"a}st{\"o} \& Nekvinda; further surveys are due to Samko
\cite{S_survey} and Mingione \cite{Min_surv}. Apart from interesting
theoretical considerations, the motivation to study such function
spaces comes from applications to fluid dynamics, image processing,
PDE and the calculus of variation.

The first concrete application arose from a model of electrorheological
fluids in \cite{RajR96} (cf.\ \cite{AM2002,AM2002_CR,Ru,Ru2} for
mathematical treatments of the model). To give the reader a feeling
for the idea behind this application we mention that an electrorheological
fluid is a so-called smart material in which the viscosity depends on the
external electric field. This
dependence is expressed through the variable exponent $p$;
specifically, the motion of the fluid is described by a
Navier--Stokes-type equation where the Laplacian $\triangle u$ is
replaced by the $p(x)$-Laplacian $\operatorname{div}(|\nabla
u|^{p(x)-2} \nabla u)$. By standard arguments, this means that
the natural energy space of the problem is $W^{1,p(\cdot)}$,
the Sobolev space of variable integrability.
For further investigations of these differential equations
see, e.g., \cite{AM4,DR,Fan07}.

More recently, an application to image restoration was proposed by Chen, Levine \&
Rao \cite{CLR,Lev}. Their model combines isotropic and total
variation smoothing. In essence, their model requires the
minimization over $u$ of the energy
\begin{align*}
\int_{\Omega} |\nabla u(x)|^{p(x)} + \lambda
|u(x)-I(x)|^2\, dx,
\end{align*}
where $I$ is given input. Recall that in the constant
exponent case, the power $p\equiv 2$ corresponds to isotropic smoothing,
whereas $p\equiv 1$ gives total variation smoothing. Hence
the exponent varies between these two extremes in
the variable exponent model.
This variational problem has an
Euler-Lagrange equation, and the solution can be found
by solving a corresponding evolutionary PDE.

Partial differential equations have also been studied from a more
abstract and general point of view in the variable exponent setting.
In analogy to the classical case, we can approach boundary value
problems through a suitable trace space, which, by definition, is a
space consisting of restrictions of functions to the boundary. For
the Sobolev space $W^{1,p(\cdot)}$, the trace space was first
characterized by first two authors by an intrinsic norm, see
\cite{DH}. In analogy with the classical case, this trace space can
be formally denoted $F^{1-1/p(\cdot)}_{p(\cdot),p(\cdot)}$, so it is
an example of a space with variable smoothness and integrability,
albeit on with a very special relationship between the two
exponents. Already somewhat earlier Almeida \& Samko \cite{AlmS06}
and Gurka, Harjulehto \& Nekvinda \cite{GHN} had extended variable
integrability Sobolev spaces to Bessel potential spaces
$W^{\alpha,p(\cdot)}$ for constant but non-integer
$\alpha$.\footnote{After the completion of this paper we learned
that Xu \cite{Xu_pp07a,Xu_pp07b} has studied Besov and
Triebel--Lizorkin spaces with variable $p$, but fixes $q$ and
$\alpha$. The results in two subsections of Section~\ref{sect:unification}
were proved independently in \cite{Xu_pp07b}.
However, most of the advantages of unification
do not occur with only $p$ variable: for instance, trace spaces
cannot be covered, and spaces of variable smoothness are not
included. Therefore Xu's work does not essentially overlap with the
results presented here.}

Along a different line of study, Leopold \cite{Leo89a,Leo89b,Leo91,Leo99}
and Leopold \& Schrohe \cite{LeoS96} studied pseudo-differential
operators with symbols of the type $\langle \xi^{m(x)}\rangle$, and
defined related function spaces of Besov-type with variable
smoothness, formally $B^{m(\cdot)}_{p,p}$. In the case $p=2$, this
corresponds to the Sobolev space $H^{m(\cdot)} = W^{m(x),2}$.
Function spaces of variable smoothness have recently been studied by
Besov \cite{Be97,Be99,Be03,Be05}. He generalized Leopold's work by
considering both Triebel--Lizorkin spaces $\Fspq{\alpha(\cdot)}{p}{q}$
and Besov spaces $\Bspq{\alpha(\cdot)}{p}{q}$ in $\Rn$. In a
recent preprint, Schneider and Schwab \cite{SS} used
$H^{m(\cdot)}(\R)$ in the analysis of certain Black--Scholes
equations. In this application the variable smoothness corresponds to
the volatility of the market, which surely should change with time.


The purpose of the present paper is to define and study a
generalized scale of Triebel--Lizorkin type spaces with variable
smoothness, $\alpha(x)$, and variable primary and secondary indices
of integrability, $p(x)$ and $q(x)$. By setting some of the indices to
appropriate values we recover all previously mentioned
spaces as special cases, except the Besov spaces (which, like in
the classical case, form a separate scale).

Apart from the value added through unification, our new space allows
treating traces and embeddings in a uniform and comprehensive
manner, rather than doing them case by case. Some particular examples
are:
\begin{itemize}
\item
The trace space of $W^{k,p(\cdot)}$ is no longer a space of the
same type. So, if we were interested in the trace space of the
trace space, the theory of \cite{DH} no longer applies,
and thus, a new theory is needed. In contrast to this, as
we show in Section~\ref{sect:traces}, the trace of a
Triebel--Lizorkin space is again a Triebel--Lizorkin space (also in
the variable indices case), hence, no such problem occurs.
\item
Our approach allows us to use
the so-called ``$r$-trick'' (cf.\ Lemma~\ref{lem:est_g}) to study
spaces with integrability in the range $(0,\infty]$, rather than in the
range $[1,\infty]$.
\item
It is well-known that the constant exponent Triebel--Lizorkin space
$F^0_{p,2}$ corresponds to the Hardy space $H^p$ when $p\in (0, 1]$.
Hardy spaces have thus far not been studied in the variable exponent
case. Therefore, our formulation opens the door to this line of
investigation.
\end{itemize}

When generalizing Triebel--Lizorkin spaces, we have several
obstacles to overcome. The main difficulty is the absence of the
vector-valued maximal function inequalities. It turns out that the
inequalities are not only missing, rather, they do not even hold in
the variable indices case (see Section~\ref{sect:multipliers}). As a
consequence of this, the H{\"o}rmander--Mikhlin multiplier theorem
does not apply in the case of variable indices.  Our solution is to
work in closer connection with the actual structure of the space
with what we call $\eta$-functions and to derive suitable estimates
directly for these functions.

The structure of the article is as follows: we first briefly
recapitulate some standard definitions and results in the next
section. In Section~\ref{sect:mainresults} we state our main results:
atomic and molecular decomposition of Triebel--Lizorkin spaces, a
trace theorem, and a multiplier theorem. In
Section~\ref{sect:unification} we show that our new scale is indeed a
unification of previous spaces, in that it includes them all as
special cases with appropriate choices of the indices. In
Section~\ref{sect:multipliers} we formulate and prove an appropriate
version of the multiplier theorem. In Section~\ref{sect:decomposition}
we give the proofs of the main decompositions theorems, and in
Section~\ref{sect:traces} we discuss the trace theorem. Finally, in
Appendix~\ref{sect:technical} we derive several technical lemmas
that were used in the other sections.


\section{Preliminaries}
\label{sect:prelim}

For $x\in \Rn$ and $r>0$ we denote by $B^n(x,r)$ the open ball in
$\Rn$ with center $x$ and radius $r$.  By $B^n$ we denote the unit
ball $B^n(0,1)$.  We use $c$ as a generic constant, i.e., a
constant whose values may change from appearance to appearance. The
inequality $f \approx g$ means that $\frac1c g \le f\le c g$ for
some suitably independent constant $c$. By $\chi_A$ we denote the
characteristic function of the set $A$. If $a\in \R$, then we use
the notation $a_+$ for the positive part of $a$, i.e., $a_+ =
\max\{0,a\}$. By $\N$ and $\N_0$ we denote the sets of positive and
non-negative integers. For $x\in \R$ we denote by $\lfloor x\rfloor$
the largest integer less than or equal to $x$.

We denote the mean-value of the integrable function $f$, defined on a
set $A$ of finite, non-zero measure, by
\begin{align*}
  \dashint\nolimits_A f(x)\, dx = \frac{1}{|A|} \int_A f(x)\, dx.
\end{align*}
The Hardy-Littlewood maximal operator $M$ is defined on
$L^1_\loc(\Rn)$ by
\begin{align*}
  M f(x) = \sup_{r>0}\, \dashint_{B^n(x,r)} \!\!\abs{f(y)}\,dy.
\end{align*}
By $\supp f$ we denote the support of the function $f$, i.e., the
closure of its zero set.


\subsection*{Spaces of variable integrability}
\label{sec:lpx}

By $\Omega\subset \Rn$ we always denote an open set. By a
\textit{variable exponent} we mean a measurable bounded function $p\colon
\Omega\to (0,\infty)$ which is bounded away from zero. For
$A\subset \Omega$ we denote $p_A^+ =\esssup_A p(x)$ and
$p_A^-=\essinf_A p(x)$; we abbreviate $p^+=p_\Omega^+$ and
$p^-=p_\Omega^-$.  We define {\it the modular} of a measurable function $f$
to be
\begin{align*}
  \varrho_{L^{p(\cdot)}(\Omega)}(f)=\int_{\Omega}|f(x)|^{p(x)}\,dx.
\end{align*}

The \emph{variable exponent Lebesgue space} $L^{p(\cdot)}(\Omega)$
consists of all measurable functions $f\colon \Omega\to\R$ for
which $\varrho_{L^{p(\cdot)}(\Omega)}(f)<\infty$.
We define {\it the Luxemburg norm} on this space by
\begin{align*} \| f\|_{L^{p(\cdot)}(\Omega)} =
\inf\big\{\lambda>0 \colon
\varrho_{L^{p(\cdot)}(\Omega)}(f/\lambda)\leqslant 1\big\},
\end{align*}
which is the Minkowski functional of the absolutely convex set
$\set{f\,:\, \varrho_{L^{p(\cdot)}(\Omega)}(f) \leqslant 1}$.  In
the case when $\Omega=\Rn$ we replace the $L^{p(\cdot)}(\Rn)$ in
subscripts simply by $p(\cdot)$, e.g.\ $\|f\|_{p(\cdot)}$ denotes
$ \|f\|_{L^{p(\cdot)}(\Rn)}$. The \emph{variable exponent Sobolev
space} $W^{1,p(\cdot)}(\Omega)$ is the subspace of
$L^{p(\cdot)}(\Omega)$ of functions $f$ whose distributional
gradient exists and satisfies $|\nabla f|\in L^{p(\cdot)}(\Omega)$.
The norm
\begin{align*}
\|f\|_{W^{1,p(\cdot)}(\Omega)}=\|f\|_{L^{p(\cdot)}(\Omega)} +
\|\nabla f\|_{L^{p(\cdot)}(\Omega)}
\end{align*}
makes $W^{1,p(\cdot)}(\Omega)$ a Banach space.

For fixed exponent spaces we of course have a very simple
relationship between the norm and the modular. In the variable
exponent case this is not so.  However, we have nevertheless the
following useful property: $\varrho_{p(\cdot)}(f) \leqslant 1$ if
and only if $\|f\|_{p(\cdot)} \leqslant 1$.  This and many other
basic results were proven in \cite{KR}.
\begin{definition}
  \label{def:loghoelder}
  Let $g \in C(\R^n)$. We say that $g$ is {\em locally
    $\log$-H{\"o}lder continuous}, abbreviated $g \in
  C^{\ln}_\loc(\R^n)$, if there exists $c_{\log}>0$ such that
  \begin{align*}
    \abs{g(x)-g(y)} \leq \frac{c_{\log}}{\ln (e + 1/\abs{x-y})}
  \end{align*}
  for all $x,y \in \R^n$.

We say that $g$ is {\em globally $\log$-H{\"o}lder continuous},
abbreviated $g \in C^{\ln}(\R^n)$, if it is locally
$\log$-H{\"o}lder continuous and there exists $g_\infty \in \R$
such that
\begin{align*}
\abs{g(x) - g_\infty} &\leq \frac{c_{\log}}{\ln(e
+ \abs{x})}
\end{align*}
for all $x \in \R^n$.
\end{definition}

Note that $g$ is globally $\log$-H{\"o}lder continuous if and only if
  \begin{align*}
    \abs{g(x)-g(y)} \leq \frac{c}{\abs{\ln \tfrac12 q(x,y)}}
  \end{align*}
  for all $x,y\in \overline{\Rn}$, where $q$ denotes the
  spherical-chordal metric (the metric inherited from a projection to the
Riemann sphere), hence the name, global $\log$-H{\"o}lder continuity.

Building on \cite{CFN} and \cite{D2} it is shown
in \cite[Theorem~3.6]{DieHHMS_pp07} that
\begin{align*}
M \colon L^{p(\cdot)}(\Rn) \hookrightarrow L^{p(\cdot)}(\Rn)
\end{align*}
is bounded if $p \in C^{\log}(\Rn)$ and $1 < p^- \leq p^+ \leq \infty$.
Global $\log$-H{\"o}lder
continuity is the best possible modulus of continuity to imply the
boundedness of the maximal operator, see \cite{CFN,PR}. However, if
one moves beyond assumptions based on continuity moduli, it is
possible to derive results also under weaker assumptions, see
\cite{Di4,Le,Ne}.


\subsection*{Partitions}

Let $\mathcal{D}$ be the collection of dyadic cubes in $\R^n$ and
denote by $\mathcal{D}^+$ the subcollection of those dyadic cubes with
side-length at most $1$.  Let $\mathcal{D}_{\nu} = \{Q \in
\mathcal{D}: \ell(Q) = 2^{-\nu} \}$.
For a cube $Q$ let $\ell(Q)$ denote the side length of $Q$ and $x_Q$
the ``lower left corner''.  For $c>0$, we let $cQ$ denote the cube
with the same center and orientation as $Q$ but with side length
$c \ell(Q)$.

The set $\mathcal{S}$ denotes the usual Schwartz space of rapidly
decreasing complex-valued functions and $\mathcal{S}'$ denotes the dual
space of tempered distributions. We denote the Fourier transform
of $\phi$ by $\hat \phi$ or $\mathcal{F} \phi$.


\begin{definition}\label{phiDef}
  We say a pair $(\phi, \Phi)$ is \textit{admissible} if $\phi, \Phi \in
  \mathcal{S}(\R^n)$ satisfy
  \begin{itemize}
  \item $\supp\hat{\phi} \subseteq \{ \xi \in
    \R^n :~\frac12 \leq |\xi| \leq 2\}$  and $|\hat{\phi}(\xi)|
    \geq c > 0$ when $\frac35 \leq |\xi| \leq \frac53$,
  \item $\supp \hat{\Phi} \subseteq \{ \xi \in \R^n :
    |\xi| \leq 2\}$ and $\ds |\hat{\Phi}(\xi)| \geq c > 0$ when $|\xi|
    \leq \tfrac53$.
  \end{itemize}
  We set $\phi_{\nu}(x) = 2^{\nu n} \phi(2^{\nu}x)$ for $\nu \in
  \N$ and $\phi_0(x)=\Phi(X)$. For $Q \in \mathcal{D}_\nu$ we set
  \begin{align*}
\phi_Q(x)=
\begin{cases}
|Q|^{1/2} \phi_{\nu}(x - x_Q) &\quad\text{if } \nu\ge 1,\\
|Q|^{1/2} \Phi(x - x_Q) &\quad\text{if } \nu = 0.
\end{cases}
\end{align*}
We define $\psi_\nu$ and $\psi_Q$ analogously.
\end{definition}

Following \cite{FraJa}, given an admissible pair $(\phi, \Phi)$ we
can select another admissible pair $(\psi, \Psi)$ such that
\begin{equation*}
  \hat{\tilde{\Phi}}(\xi) \cdot \hat{\Psi}(\xi) + \sum_{\nu \geq 1}
  {\hat{\tilde{\phi}}(2^{-\nu}\xi)} \cdot \hat{\psi}(2^{-\nu}\xi) = 1
  \quad\text{for all }\xi.
\end{equation*}
Here, $\tilde{\Phi}(x) = \overline{\Phi(-x)}$ and similarly for
$\tilde\phi$.

For each $f \in \cS'(\R^n)$ we define the ({\it inhomogeneous})
$\phi$-{\it transform} $S_{\phi}$ as the map taking ${f}$
to the sequence $(S_{\phi} {f})_{Q\in \mathcal{D}^+}$
by setting $\ds (S_{\phi} {f}\,)_Q = \ip{{f}}{\phi_Q}$.
Here, $\skp{{\cdot}}{\cdot}$ denotes the usual inner
product on $L^2(\R^n;\C)$.
For later purposes note that
$(S_\phi f)_Q = |Q|^{1/2} \tilde{\phi}_{\nu} \ast f(2^{-\nu} k)$
for $l(Q) = 2^{-\nu} < 1$ and
$(S_\phi f)_Q = |Q|^{1/2} \tilde{\Phi} \ast f(2^{-\nu} k)$
for $l(Q) = 1$.

The {\it inverse} ({\it inhomogeneous}) $\phi$-{\it transform}
$T_{\psi}$ is the map taking a sequence $s=\{s_Q\}_{l(Q) \leq
1}$ to $\ds T_{\psi} s = \sum_{l(Q)=1} s_Q \Psi_Q +\sum_{l(Q)
<1} s_Q \psi_Q$. We have the following identity for $f \in
\cS'(\R^n)$:
\begin{equation}
  f = \sum_{Q\in \mathcal{D}_0} \ip{f}{\Phi_Q} \Psi_Q + \sum_{\nu=1}^{\infty}
  \sum_{Q\in \mathcal{D}_\nu} \ip{f}{\phi_Q} \psi_Q.
  \label{inhomFJdecomp}
\end{equation}
Note that we consider all distributions in $\cS'(\R^n)$ (rather than
$\cS'/\mathcal{P}$ as in the homogeneous case), since $\hat{\Phi}(0)
\neq 0$.

Using the admissible functions $(\phi,\Phi)$ we can define the
norms
  \begin{align*}
    \norm{f}_{\Fspq{\alpha}{p}{q}} =
    \Bignorm{ \bignorm{ 2^{\nu \alpha}\, \phi_\nu \ast
        f}_{l^q} }_{L^p}
\quad\text{and}\quad
    \norm{f}_{\Bspq{\alpha}{p}{q}} =
    \Bignorm{ \bignorm{ 2^{\nu \alpha}\, \phi_\nu \ast
        f }_{L^p}}_{l^q},
  \end{align*}
  for constants $p,q\in (0,\infty)$ and $\alpha\in \R$.
  The Triebel--Lizorkin space $\Fspq{\alpha}{p}{q}$
  and the Besov space $\Bspq{\alpha}{p}{q}$ consists of
  distributions $f\in \cS'$ for which $ \norm{f}_{\Fspq{\alpha}{p}{q}}<\infty$
  and $\norm{f}_{\Bspq{\alpha}{p}{q}}<\infty$, respectively.
  The classical theory of these spaces is presented for instance in the
  books of Triebel \cite{Tr1,Tr2,Tr3}.
  The discrete representation as sequence spaces through the $\phi$-transform
is due to Frazier and Jawerth \cite{FJ1,FraJa}.
Recently, anisotropic and weighted versions of these
  spaces have been studied by many people, see, e.g.,
  Bownik and Ho \cite{BowH06}, Frazier and Roudenko \cite{Rou03,FraR05},
  Kühn, Leopold, Sickel and Skrzypczak \cite{KuhLSS07}, and
  the references therein.
  We now move on to generalizing these definitions to the variable
  index case.


\section{Statement of the main results}
\label{sect:mainresults}

In this section we introduce the main tool of this paper,
a decomposition of the Triebel--Lizorkin space into molecules or
atoms and state other important results. Section~\ref{sect:unification}
contains further main results: there we show that
previously studied spaces are indeed included in our scale.
The proofs of the results from this section constitute much of the
remainder of this article.

Throughout the paper we use the following

\begin{standingassumption}
  \label{ass:pqalpha}
  We assume that $p,q$ are positive functions on $\R^n$ such that
  $\frac{1}{p}, \frac{1}{q} \in C^{\ln}(\Rn)$. This implies, in particular,
  $0 < p^- \leq p^+ < \infty$ and $0 < q^- \leq q^+ < \infty$.  We
  also assume that $\alpha \in C^{\ln}_\loc(\Rn) \cap L^\infty(\Rn)$ with
  $\alpha\geq 0$ and that $\alpha$ has a limit at infinity.
\end{standingassumption}

One of the central classical tools that we are missing in the variable
integrability setting is a general multiplier theorem of
Mikhlin--H{\"o}rmander type. We show in Section~\ref{sect:multipliers}
that a general theorem does not hold, and instead prove the
following result which is still sufficient to work with
Triebel--Lizorkin spaces.

For a family of functions $f_\nu\,:\,\R^n \to \R$, $\nu \geq
0$, we define
\begin{align*}
  \bignorm{ f_\nu(x)}_{\lqxnu} &= \bigg( \sum_{\nu \geq 0}
  \abs{f_\nu(x)}^{q(x)}\bigg)^{\frac{1}{q(x)}}.
\end{align*}
Note that this is just an ordinary discrete Lebesgue space, since $q(x)$ does
not depend on~$\nu$. The mapping $x\mapsto\norm{f_\nu(x)}_{\lqxnu}$ is
a function of~$x$ and can be measured in $L^{p(\cdot)}$. We write
$L^{p(\cdot)}_x$ to indicate that the integration variable is~$x$.
We define
\begin{equation}
 \label{E:nu}
\eta_m(x)=(1+|x|)^{-m} \quad \text{and} \quad
\eta_{\nu,m}(x) = 2^{n\nu} \eta_m(2^\nu x).
\end{equation}

\begin{theorem}
  \label{thm:eta}
  Let $p,q \in C^{\ln}(\R^n)$ with $1 < p^- \leq p^+ < \infty$
  and $1 < q^- \leq q^+ < \infty$. Then the inequality
  \begin{align*}
    \Bignorm{ \bignorm{ \eta_{\nu,m} \ast f_\nu }_{\lqxnu} }_{\Lpdotx}
    &\leq c\, \Bignorm{ \bignorm{ f_\nu }_{\lqxnu} }_{\Lpdotx}
  \end{align*}
  holds for every sequence $\set{f_\nu}_{\nu \in \N_0}$ of
  $L^1_\loc$-functions and constant $m>n$.
\end{theorem}

\begin{definition}
  \label{def:Fbetapq}
  Let $\phi_\nu$,  $\nu \in \N_0$, be as in
  Definition~\ref{phiDef}.  The Triebel--Lizorkin space
  $\Fspq{\alpha(\cdot)}{p(\cdot)}{q(\cdot)}(\R^n)$ is defined to be
  the space of all distributions $f \in \mathcal{S}'$ with
  $\norm{f}_{\Fspq{\alpha(\cdot)}{p(\cdot)}{q(\cdot)}} < \infty$,
  where
  \begin{align*}
    \norm{f}_{\Fspq{\alpha(\cdot)}{p(\cdot)}{q(\cdot)}} &:=
    \Bignorm{ \bignorm{ 2^{\nu \alpha(x)}\, \phi_\nu \ast
        f(x)}_{\lqxnu} }_{\Lpdotx}.
  \end{align*}
  In the case of $p=q$ we use the notation
  $\Fsp{\alpha(\cdot)}{p(\cdot)}(\R^n) :=
  \Fspq{\alpha(\cdot)}{p(\cdot)}{q(\cdot)}(\R^n)$.
\end{definition}

Note that, \textit{a priori}, the function space depends on the
choice of admissible functions $(\phi,\Phi)$. One of the main
purposes of this paper is to show that, up to equivalence of
norms, every pair of admissible functions produces the same space.

In the classical case it has proved very useful to express the
Triebel--Lizorkin norm in terms of two sums, rather than a sum and
an integral, thus, giving rise to discrete Triebel--Lizorkin spaces
$\fspq{\alpha}{p}{q}$. Intuitively, this is achieved by viewing the
function as a constant on dyadic cubes. The size of the appropriate
dyadic cube varies according to the level of smoothness.

We next present a formulation of the Triebel--Lizorkin norm which is
similar in spirit.

For a sequence of real numbers $\{s_Q\}_Q$ we define
\begin{align*}
  \bignorm{\set{s_Q}_{Q}}_{\fspq{\alpha(\cdot)}{p(\cdot)}{q(\cdot)}} &:=
  \Biggnorm{ \biggnorm{  2^{\nu\alpha(x)} \sum_{Q \in \mathcal{D}_\nu}
      \abs{s_Q}\, \abs{Q}^{-\frac{1}{2}}\, \chi_Q }_{\lqxnu}
  }_{\Lpdotx}.
\end{align*}
The space $\fspq{\alpha(\cdot)}{p(\cdot)}{q(\cdot)}$ consists of
all those sequences $\{s_Q\}_Q$ for which this norm is finite.
We are ready to state our first decomposition result, which
says that $S_\phi \colon \Fspq{\alpha(\cdot)}{p(\cdot)}{q(\cdot)}
\hookrightarrow \fspq{\alpha(\cdot)}{p(\cdot)}{q(\cdot)}$ is a bounded operator.

\begin{theorem}
  \label{thm:Sphi_bnd}
  If $p$, $q$ and $\alpha$ are as in the Standing Assumptions, then
  \begin{align*}
    \norm{S_\phi f}_{\fspq{\alpha(\cdot)}{p(\cdot)}{q(\cdot)}} &\leq
    c\,\norm{f}_{\Fspq{\alpha(\cdot)}{p(\cdot)}{q(\cdot)}}.
  \end{align*}
\end{theorem}

If we have a sequence $\{s_Q\}_Q$, then we can easily construct
a candidate Triebel--Lizorkin function by taking the
weighted sum with certain basis functions, $\sum s_Q m_Q$. Obviously,
certain restrictions are necessary on the functions $m_Q$ in order for this to work.
We therefore make the following definitions:

\begin{definition}
  \label{def:molecule}
  Let $\nu \in \N_0$, $Q \in \mathcal{D}_\nu$ and $k\in \Z$, $l\in \N_0$ and
$M\ge n$. A function $m_Q$ is said to be a \textit{$(k,l,M)$-smooth molecule
  for $Q$} if it satisfies the following conditions for some $m > M$:
  \begin{enumerate}[label=(M\arabic{*})]
  \item \label{itm:mol1} if $\nu > 0$, then $\displaystyle \int_{\R^n}
    x^{\gamma} m_Q (x) \, dx = 0$ for all $|\gamma| \leq k$; and
  \item \label{itm:mol2} $|D^{\gamma} m_Q(x)| \leq 2^{\abs{\gamma} \nu}
    \abs{Q}^{1/2} \,\eta_{\nu,m}(x+x_Q)$ for all multi-indices
    $\gamma\in \N_0^n$ with $|\gamma| \leq l$.
  \end{enumerate}
  The conditions \ref{itm:mol1} and \ref{itm:mol2} are called
the {\emph {moment}} and {\emph {decay conditions}}, respectively.
\end{definition}

Note that \ref{itm:mol1} is vacuously true if $k<0$.
When $M=n$, this definition is a special case of the definition
given in \cite{FraJa} for molecules. The difference is that
we consider only $k$ and $l$ integers, and $l$ non-negative.
In this case two of the four conditions given in \cite{FraJa}
are vacuous.

\begin{definition}
  Let $K,L\colon \Rn\to\R$ and $M>n$. The family $\{m_Q\}_Q$ is said to be a
  \textit{family of $(K,L,M)$-smooth
molecules} if $m_Q$ is $(\lfloor K_Q^- \rfloor ,\lfloor L_Q^-\rfloor ,M)$-smooth for every $Q\in \mathcal{D}^+$.
\end{definition}

\begin{definition}
\label{def:molForF}
We say that $\set{m_Q}_Q$ is a \textit{family of smooth molecules for
$\Fspq{\alpha(\cdot)}{p(\cdot)}{q(\cdot)}$} if it is a family of
$(N+\epsilon,\alpha+1+\epsilon,M)$-smooth molecules, where
\begin{align*}
N(x):= \frac{n}{\min \set{1,p(x),q(x)}} - n -  \alpha(x),
\end{align*}
for some constant $\epsilon>0$, and $M$ is a sufficiently large constant.
\end{definition}

The number $M$ needs to be chosen sufficiently large, for instance
\[
2\frac{n+c_\text{log}(\alpha)}{\min\{1,p^-,q^-\}}
\]
will do, where
$c_\text{log}(\alpha)$ denotes the $\log$-H\"older continuity constant
of $\alpha$. Since $M$ can be fixed depending on the parameters we
will usually omit it from our notation of molecules.

Note that the functions $\phi_Q$ are smooth molecules for arbitrary
indices. Also note that compared to the classical case
we assume the existance of $1$ more derivative (rounded down)
for smooth molecules for $\Fspq{\alpha(\cdot)}{p(\cdot)}{q(\cdot)}$.
We need the assumption for technical reasons (cf.\ Lemma~\ref{lem:Sphi_inv}).
However, we think the additional assumptions are inconsequential;
for instance the trace result (Theorem~\ref{thm:trace}),
and indeed any result based on atomic decomposition,
can still be proven in an optimal form.

\begin{theorem}
  \label{thm:Sphi_inv}
  Let the functions $p$, $q$, and $\alpha$ be as in the Standing Assumptions.
Suppose that $\{m_Q\}_Q$ is a family of smooth molecules for
$\Fspq{\alpha(\cdot)}{p(\cdot)}{q(\cdot)}$ and that $\{s_Q\}_Q \in
  \fspq{\alpha(\cdot)}{p(\cdot)}{q(\cdot)}$. Then
  \begin{align*}
  \norm{f}_{\Fspq{\alpha(\cdot)}{p(\cdot)}{q(\cdot)}} \leq c\,
  \norm{\set{s_Q}_Q}_{\fspq{\alpha(\cdot)}{p(\cdot)}{q(\cdot)}},
  \quad\text{where}\quad
  f = \sum_{\nu \geq 0}
  \sum_{Q \in \mathcal{D}_\nu} s_Q m_Q.
  \end{align*}
\end{theorem}

Theorems~\ref{thm:Sphi_bnd} and \ref{thm:Sphi_inv} yield an
isomorphism between $\Fspq{\alpha(\cdot)}{p(\cdot)}{q(\cdot)}$ and
a subspace of $\fspq{\alpha(\cdot)}{p(\cdot)}{q(\cdot)}$ via the $S_\phi$
transform:

\begin{corollary}
  \label{cor:F=f}
  If the functions $p$, $q$, and $\alpha$ are as in the Standing Assumptions,
  then
  \begin{align*}
    \norm{f}_{\Fspq{\alpha(\cdot)}{p(\cdot)}{q(\cdot)}} \approx
    \norm{S_\phi f}_{\fspq{\alpha(\cdot)}{p(\cdot)}{q(\cdot)}}
  \end{align*}
  for every $f \in \Fspq{\alpha(\cdot)}{p(\cdot)}{q(\cdot)}(\Rn)$.
\end{corollary}

With these tools we can prove that the space
$\Fspq{\alpha(\cdot)}{p(\cdot)}{q(\cdot)}(\R^n)$ is well-defined.

\begin{theorem}
  The space $\Fspq{\alpha(\cdot)}{p(\cdot)}{q(\cdot)}(\R^n)$ is
  well-defined, i.e.,\ the definition does not depend on the choice of
  the functions $\phi$ and $\Phi$ satisfying the conditions of
  Definition~\ref{phiDef}, up to the equivalence of norms.
\end{theorem}

\begin{proof}
  Let ${\tilde\phi}_\nu$ and $\phi_\nu$ be different basis functions as in
  Definition~\ref{phiDef}.  Let $\norm{\cdot}_{\tilde\phi}$ and
  $\norm{\cdot}_\phi$ denote the corresponding norms of
  $\Fspq{\alpha(\cdot)}{p(\cdot)}{q(\cdot)}$. By symmetry, it suffices to prove
  $\norm{f}_{\tilde\phi} \le c\,  \norm{f}_\phi$ for all $f \in \mathcal{S}'$.
  Let $\norm{f}_\phi < \infty$. Then by~\eqref{inhomFJdecomp} and
  Theorem~\ref{thm:Sphi_bnd} we have $f = \sum_{Q \in \mathcal{D}^+}
  (S_\phi f)_Q \psi_Q$ and
$\norm{S_\phi f}_{\fspq{\alpha(\cdot)}{p(\cdot)}{q(\cdot)}} \leq
    c\,\norm{f}_\phi$.
Since $\{\psi_Q\}_Q$ is a family of smooth molecules,
  $\norm{f}_{\tilde\phi} \leq c\,
    \norm{S_\phi f}_{\fspq{\alpha(\cdot)}{p(\cdot)}{q(\cdot)}}$
by Theorem~\ref{thm:Sphi_inv}, which completes the proof.
\end{proof}

It is often convenient to work with compactly supported basis functions.
Thus, we say that the molecule $a_Q$ concentrated on~$Q$ is
an \textit{atom} if it satisfies $\supp a_Q \subset 3Q$. The downside
of atoms is that we need to chose a new set of them for each
function $f$ that we represent. Note that this coincides
with the definition of atoms in \cite{FraJa} in the case when
$p$, $q$ and $\alpha$ are constants.

For atomic decomposition we have the following result.

\begin{theorem}
  \label{thm:atom_Fspxq}
  Let the functions $p$, $q$, and $\alpha$ be as in the Standing Assumptions and let
  $f\in \Fspq{\alpha(\cdot)}{p(\cdot)}{q(\cdot)}$. Then there exists a
  family of smooth atoms $\set{a_Q}_Q$ and a sequence of coefficients
  $\set{t_Q}_Q$ such that
  \begin{align*}
    f=\sum_{Q \in \mathcal{D}^+} t_Q a_Q \ \text{ in }\mathcal{S}'
    \quad\text{and}\quad
    \bignorm{\set{t_Q}_Q}_{\fspq{\alpha(\cdot)}{p(\cdot)}{q(\cdot)}}
    &\approx\,\norm{f}_{\Fspq{\alpha(\cdot)}{p(\cdot)}{q(\cdot)}}.
  \end{align*}
  Moreover, the atoms can be chosen to satisfy conditions
  \ref{itm:mol1} and \ref{itm:mol2} in Definition~\ref{def:molecule}
  for arbitrarily high, given order.
\end{theorem}

If the maximal operator is bounded and $1 < p^- \leq p^+ < \infty$,
then it follows easily that $C^\infty_0(\Rn)$ (the space of smooth
functions with compact support) is dense in $W^{1,p(\cdot)}(\Rn)$,
since it is then possible to use convolution. However, density can
be achieved also under more general circumstances, see
\cite{FanWZ06,Ha2,Zhi05}. Our standing assumptions are strong
enough to give us density directly:

\begin{corollary}
\label{cor:density}
Let the functions $p$, $q$, and $\alpha$ be as in the Standing Assumptions.
Then $C^\infty_0(\Rn)$ is dense in
$\Fspq{\alpha(\cdot)}{p(\cdot)}{q(\cdot)}(\Rn)$.
\end{corollary}

Another consequence of our atomic decomposition is the analogue of
the standard trace theorem. Since its proof is much more involved, we present it
in Section~\ref{sect:traces}. Note that the assumption
$\alpha- \frac{1}{p} - (n-1) \Big(\frac1p - 1\Big)_+ > 0$
is optimal also in the constant smoothness and integrability
case, cf.~\cite[Section~5]{FJ1}

\begin{theorem}
  \label{thm:trace}
  Let the functions $p$, $q$, and $\alpha$ be as in
  the Standing Assumptions. If
  \begin{align*}
  \alpha- \frac{1}{p} - (n-1) \bigg(\frac1p - 1\bigg)_+ > 0,
\quad\text{then}\quad
    \trace \Fspq{\alpha(\cdot)}{p(\cdot)}{q(\cdot)}(\Rn) =
    \Fsp{\alpha(\cdot)- \frac{1}{p(\cdot)}}{p(\cdot)}(\R^{n-1}).
  \end{align*}
\end{theorem}


\section{Special cases}  
\label{sect:unification}

In this section we show how the Triebel--Lizorkin scale
$\Fspq{\alpha(\cdot)}{p(\cdot)}{q(\cdot)}$ includes as special cases
previously studied spaces with variable differentiability or
integrability.

\subsection*{Lebesgue spaces}
\addcontentsline{toc}{section}{~\hspace{1cm}Lebesgue spaces}
We begin with the variable exponent Lebesgue spaces from
Section~\ref{sec:lpx}, which were originally introduced by Orlicz
in~\cite{Orl31}. We show that $\Fspq{0}{p(\cdot)}{2} \cong
L^{p(\cdot)}$ under suitable assumptions on $p$. We use an
extrapolation result for $L^{p(\cdot)}$.
Recall, that a weight $\omega$ is in the Muckenhoupt class $A_1$ if
$M\omega \leq K\, \omega$ for some such $K>0$. The smallest $K$ is the $A_1$ constant of
$\omega$.

\begin{lemma}[Theorem 1.3, \cite{CFMP}]
  \label{lem:extrapol}
  Let $p \in C^{\ln}(\R^n)$ with $1< p^- \leq p^+< \infty$ and let
  $\mathcal{G}$ denote a family of tuples $(f, g)$ of measurable
  functions on $\R^n$.  Suppose that there exists a constant $r_0 \in (0, p^-)$
  so that
  \begin{align*}
    \bigg(\int_{\R^n} \abs{f(x)}^{r_0}
    \omega(x)\,dx\bigg)^{\frac{1}{r_0}} &\leq c_0 \bigg( \int_{\R^n}
    \abs{g(x)}^{r_0} \omega(x)\,dx \bigg)^{\frac{1}{r_0}}
  \end{align*}
  for all $(f, g) \in \mathcal{G}$ and every weight $\omega \in A_1$,
  where $c_0$ is independent of $f$ and $g$ and depends on
  $\omega$ only via its $A_1$-constant. Then
  \begin{align*}
    \norm{ f}_{L^{p(\cdot)}(\R^n)} &\leq
    c_1\,\norm{g}_{L^{p(\cdot)}(\R^n)}
  \end{align*}
  for all $(f,g) \in \mathcal{G}$ with $\norm{ f}_{L^{p(\cdot)}(\R^n)} < \infty$.
\end{lemma}

\begin{theorem}
  \label{thm:Fspq=Lpx}
  Let $p \in C^{\ln}(\R^n)$ with $1< p^- \leq p^+< \infty$. Then
  $L^{p(\cdot)}(\R^n) \cong \Fspq{0}{p(\cdot)}{2}(\R^n)$. In
  particular,
  \begin{align*}
    \norm{f}_{L^{p(\cdot)}(\R^n)} \approx \bignorm{ \norm{\phi_\nu *
        f}_{l^2_\nu} }_{L^{p(\cdot)}(\R^n)}
  \end{align*}
  for all $f \in L^{p(\cdot)}(\R^n)$.
\end{theorem}

\begin{proof}
  Since $C^\infty_0(\R^n)$ is dense in $L^{p(\cdot)}(\R^n)$
  (see \cite{KR}) and also in $\Fspq{0}{p(\cdot)}{2}(\R^n)$ by
  Corollary~\ref{cor:density}, it suffices to prove the claim for all
  $f \in C^\infty_0(\R^n)$.  Fix $r \in (1, p^-)$. Then
  \begin{align*}
    \bignorm{ \norm{\phi_\nu * f}_{l^2_\nu}
    }_{L^{r_0}(\R^n;\omega)} &\approx
    \,\norm{f}_{L^{r_0}(\R^n;\omega)},
  \end{align*}
  for all $\omega \in A_1$ by \cite[Theorem~1]{Kur80}, where the constant depends
  only on the $A_1$-constant of the weight $\omega$, so the
  assumptions of Lemma~\ref{lem:extrapol} are satisfied. Applying the
  lemma with $\mathcal{G}$ equal to either
  \begin{align*}
    \bigset{ \big( \norm{ \phi_\nu*f}_{l^2_\nu}, f\big) \,:\, f \in
      C^\infty_0(\Omega)} \quad\text{or}\quad \set{ \big(f,\norm{
        \phi_\nu*f}_{l^2_\nu}\big) \,:\, f \in C^\infty_0(\Omega)}
  \end{align*}
  completes the proof.
\end{proof}

Theorem~\ref{thm:Fspq=Lpx} generalizes the equivalence of
$L^p(\R^n) \cong \Fspq{0}{p}{q}$ for constant~$p \in (1,\infty)$ to
the setting of variable exponent Lebesgue spaces. If $p \in (0,1]$,
then the spaces $L^p(\R^n)$ have to be replaced by the Hardy spaces
$h^p(\R^n)$. This suggests the following definition:

\begin{definition}
  \label{def:hardypx}
  Let $p \in C^{\ln}(\R^n)$ with $0< p^- \leq p^+< \infty$. Then we
  define the {\em variable exponent Hardy space}
  $h^{p(\cdot)}(\R^n)$ by $h^{p(\cdot)}(\Rn) := \Fspq{0}{p(\cdot)}{2}(\Rn)$.
\end{definition}

The investigation of this space is left for future research.

\subsection*{Sobolev and Bessel spaces}
\addcontentsline{toc}{section}{~\hspace{1cm}Sobolev and Bessel spaces}
We move on to Bessel potential spaces with variable
integrability, which have been independently introduced by Almeida \&
Samko~\cite{AlmS06}
and Gurka, Harjulehto \& Nekvinda~\cite{GHN}.
This scale includes also the variable
exponent Sobolev spaces $W^{k,p(\cdot)}$.

In the following let $\mathcal{B}^\sigma$ denote the Bessel potential
operator $\mathcal{B}^{\sigma} = \mathcal{F}^{-1}
(1+\abs{\xi}^2)^{-\frac{\sigma}{2}} \mathcal{F}$ for $\sigma \in
\R$.  Then the {\em variable exponent Bessel potential space} is
defined by
\begin{align*}
  \mathcal{L}^{\alpha,p(\cdot)}(\R^n) &:= \mathcal{B}^\alpha \big(
  L^{p(\cdot)}(\R^n)\big) = \set{\mathcal{B}^\alpha g\,:\, g
    \in L^{p(\cdot)}(\R^n)}
\end{align*}
equipped with the norm $\norm{g}_{\mathcal{L}^{\alpha,p(\cdot)}}
:= \norm{\mathcal{B}^{-\alpha} g}_{p(\cdot)}$.
It was shown independently in~\cite[Corollary~6.2]{AlmS06}
and~\cite[Theorem~3.1]{GHN} that $\mathcal{L}^{k,p(\cdot)}(\R^n)
\cong W^{k,p(\cdot)}(\R^n)$ for $k \in \N_0$ when $p\in C^{\ln}(\Rn)$
with $1<p^- \le p^+ <\infty$.

We will show that $\mathcal{L}^{\alpha,p(\cdot)}(\R^n) \cong
\Fspq{\alpha}{p(\cdot)}{2}(\R^n)$ under suitable assumptions on~$p$
for $\alpha \geq 0$ and that $\mathcal{L}^{k,p(\cdot)}(\R^n) \cong
W^{k,p(\cdot)}(\R^n) \cong \Fspq{k}{p(\cdot)}{2}(\R^n)$ for $k
\in \N_0$. It is clear by the definition of
$\mathcal{L}^{\alpha,p(\cdot)}(\R^n)$ that $\mathcal{B}^\sigma$
with $\sigma \geq 0$ is an isomorphism between
$\mathcal{L}^{\alpha,p(\cdot)}(\R^n)$ and
$\mathcal{L}^{\alpha+\sigma,p(\cdot)}(\R^n)$, i.e.,\ it has a lifting property.
Therefore, in view of Theorem~\ref{thm:Fspq=Lpx} and
$\mathcal{L}^{0,p(\cdot)}(\R^n) = L^{p(\cdot)}(\R^n) \cong
\Fspq{0}{p(\cdot)}{2}(\Rn)$, we will complete the circle by proving a
lifting property for the scale
$\Fspq{\alpha(\cdot)}{p(\cdot)}{q(\cdot)}(\Rn)$.

\begin{lemma}[Lifting property]
  \label{lem:BesselIso}
  Let $p$, $q$, and $\alpha$ be as in the Standing Assumptions and
  $\sigma \geq 0$. Then the Bessel potential operator
  $\mathcal{B}^{\sigma}$ is an isomorphism between
  $\Fspq{\alpha(\cdot)}{p(\cdot)}{q(\cdot)}$ and
  $\Fspq{\alpha(\cdot)+\sigma}{p(\cdot)}{q(\cdot)}$.
\end{lemma}

\begin{proof}
  Let $f \in \Fspq{\alpha(\cdot)}{p(\cdot)}{q(\cdot)}$.
We know that $\{\phi_Q\}$ is a family of smooth molecules,
thus, by Theorem~\ref{thm:Sphi_bnd}
  \begin{align*}
    \bignorm{ \set{s_Q}_Q }_{\fspq{\alpha(\cdot)}{p(\cdot)}{q(\cdot)}}
    \approx \bignorm{f}_{\Fspq{\alpha(\cdot)}{p(\cdot)}{q(\cdot)}},
  \end{align*}
where $f = \sum_{Q \in \mathcal{D}^+} s_Q\, \phi_Q$.
Therefore,
  \begin{align*}
    \mathcal{B}^{\sigma} f = \sum_{Q \in \mathcal{D}^+} s_Q\,
    \mathcal{B}^{\sigma} \phi_Q = \sum_{Q \in \mathcal{D}^+}
\underbrace{2^{-\nu \sigma}\, s_Q}_{\displaystyle =: s_Q'}
\, \underbrace{2^{\nu \sigma} \mathcal{B}^{\sigma} \phi_Q}_{\displaystyle =:\phi_Q'}.
  \end{align*}

Let us check that $\{K \phi_Q'\}_Q$ is a family of smooth molecules
of an arbitrary order for a suitable constant $K$.
Let $Q\in \mathcal{D}^+$. Without loss
of generality we may assume that $x_Q=0$. Then
\begin{align*}
\widehat{\phi_Q'}(\xi) =
\frac{2^{\nu \sigma}\widehat{\phi_Q}(\xi)}{(1+|\xi|)^\sigma}
=
\frac{2^{\nu \sigma} |Q|^{1/2} \hat \phi(2^{-\nu} \xi)}{(1+|\xi|)^\sigma}.
\end{align*}
Since $\hat \phi$ has support in the annulus $B^n(0,2)\setminus B^n(0,1/2)$,
it is clear that $\widehat{\phi_Q'}\equiv 0$ in a neighborhood of
the origin when $l(Q)<1$, so the family satisfies the moment
condition in Definition~\ref{def:molecule} for an arbitrarily high order.

Next we consider the decay condition for molecules.
Let $\mu\in \N_0^n$ be a multi-index with $|\mu|=m$.
We estimate
\begin{align*}
\big| D^{\mu}_\xi \widehat{\phi_Q'}(\xi)\big|
& \le
2^{\nu \sigma} |Q|^{1/2} \bigg|D^{\mu}_\xi \bigg[\frac{\hat \phi(2^{-\nu} \xi)}{(1+|\xi|)^\sigma}\bigg]\bigg| \\
& = |Q|^{1/2} 2^{-\nu  m}
\bigg|D^{\mu}_\zeta \bigg[\frac{\hat \phi(\zeta)}{(2^{-\nu}+|\zeta|)^\sigma}\bigg]\bigg| \\
& \le c\,
|Q|^{1/2} 2^{-\nu  m}
\big|D^{\mu}_\zeta \big[\hat \phi(\zeta) |\zeta|^{-\sigma}\big]\big|,
\end{align*}
where $\zeta = 2^{-\nu}\xi$ and we used that the support of $\hat\phi$
lies in the annulus $B^n(0,2)\setminus B^n(0,1/2)$ for the last estimate.
Define
\begin{align*}
K_m = \sup_{|\mu| = m, \zeta\in \Rn} 2^{-\nu  m}
\big|D^{\mu}_\zeta \big[\hat \phi(\zeta) |\zeta|^{-\sigma}\big] \big|.
\end{align*}
Since $\sigma\ge 0$ and $\hat \phi$ vanishes in a neighborhood
of the origin, we conclude that $K_m<\infty$ for every $m$.
From the estimate
\begin{equation*}
|x^\mu \psi(x)| = c \,\bigg| \int_\Rn (-1)^m D^\mu_\xi
\hat\psi(\xi)\, e^{i\, x\cdot \xi}\, d\xi\bigg| \le c\, \big|
\mbox{supp}\, \hat \psi\big|\, \sup_\xi |D^\mu_\xi \hat \psi(\xi)|,
\end{equation*}
we conclude that
\begin{equation*}
|x|^m\, |\phi_Q'(x)| \le c\,  2^{\nu n}  |Q|^{1/2} 2^{-\nu m} K_m
\quad\text{and}\quad
|\phi_Q'(x)| \le c\,  2^{\nu n}  |Q|^{1/2} K_0.
\end{equation*}
Multiplying the former of the two inequalities
by $2^{\nu  m}$ and adding it to the latter gives
\begin{align*}
\big(1  + 2^{\nu  m}|x|^m\big) \, |\phi_Q'(x)|
\le c\,  2^{\nu n} |Q|^{1/2} (K_0 + K_m).
\end{align*}
Finally, this implies that
\begin{align*}
|\phi_Q'(x)| \le c\,
\frac{2^{\nu n}}{(1 + 2^{\nu}|x|)^m} |Q|^{1/2} (K_0 + K_m)
= |Q|^{1/2} (K_0 + K_m) \eta_{\nu,m}(x),
\end{align*}
from which we conclude that the family $\{ K \phi_Q'\}_Q$
satisfy the decay condition when $K\le(|Q|^{1/2} (K_0 + K_m) )^{-1}$.
A similar argument yields the decay condition for $D_x^\mu \phi_Q'$.

Since $\{ K \phi_Q'\}_Q$ is a family of smooth molecules for
$\Fspq{\alpha(\cdot)+\sigma}{p(\cdot)}{q(\cdot)}$,
we can apply Theorem~\ref{thm:Sphi_inv} to conclude that
  \begin{align*}
    \bignorm{\mathcal{B}^{\sigma}
      f}_{\Fspq{\alpha(\cdot)+\sigma}{p(\cdot)}{q(\cdot)}} &\le c\,
    \bignorm{ \set{s_Q'/K}_Q
    }_{\fspq{\alpha(\cdot)+\sigma}{p(\cdot)}{q(\cdot)}} \le c\,
    \bignorm{ \set{ s_Q}_Q
    }_{\fspq{\alpha(\cdot)}{p(\cdot)}{q(\cdot)}} \approx
    \bignorm{f}_{\Fspq{\alpha(\cdot)}{p(\cdot)}{q(\cdot)}}.
  \end{align*}
  The reverse inequality is handled similarly. 
\end{proof}

\begin{theorem}
  \label{thm:F=Bessel}
  Let $p \in C^{\ln}(\R^n)$ with $1 < p^- \leq p^+ < \infty$ and
  $\alpha \in [0,\infty)$. Then $\Fspq{\alpha}{p(\cdot)}{2}(\R^n)
  \cong \mathcal{L}^{\alpha,p(\cdot)}(\R^n)$. If $k \in \N_0$,
  then $\Fspq{k}{p(\cdot)}{2}(\R^n) \cong
   W^{k,p(\cdot)}(\R^n)$.
\end{theorem}

\begin{proof}
Suppose that $f\in \Fspq{\alpha}{p(\cdot)}{2}(\R^n)$.
By Lemma~\ref{lem:BesselIso}, $\mathcal{B}^{-\alpha} f \in
\Fspq{0}{p(\cdot)}{2}(\R^n)$, so
we conclude by Theorem~\ref{thm:Fspq=Lpx}
that $\mathcal{B}^{-\alpha} f \in L^{p(\cdot)}(\Rn)
= \mathcal{L}^{0,p(\cdot)}(\R^n)$.
Then it follows by the definition of the Bessel space that
$f = \mathcal{B}^{\alpha} [\mathcal{B}^{-\alpha} f] \in
\mathcal{L}^{\alpha,p(\cdot)}(\R^n)$.
The reverse inclusion follows by reversing these steps.

The claim regarding the Sobolev spaces follows
  from this and the equivalence $\mathcal{L}^{k,p(\cdot)}(\R^n)
\cong W^{k,p(\cdot)}(\R^n)$ for $k \in \N_0$
(see \cite[Corollary~6.2]{AlmS06} or~\cite[Theorem~3.1]{GHN}).
\end{proof}

\subsection*{Spaces of variable smoothness}
\addcontentsline{toc}{section}{~\hspace{1cm}Spaces of variable smoothness}

Finally, we come to spaces of variable smoothness as introduced
by Besov \cite{Be97}, following Leopold \cite{Leo87}.
Let $p,q \in (1,\infty)$ and let $\alpha \in
C^{\ln}_\loc(\Rn) \cap L^\infty(\Rn)$ with $\alpha\geq 0$. Then Besov defines
the following spaces of variable smoothness
\begin{align*}
  \Fspq{\alpha(\cdot),\text{Besov}}{p}{q}(\R^n) &:= \Bigset{ f \in
    L^p_{\loc}(\R^n)\,:\,
    \bignorm{f}_{\Fspq{\alpha(\cdot),\text{Besov}}{p}{q}} < \infty},
  \\
  \norm{f}_{\Fspq{\alpha(\cdot),\text{Besov}}{p}{q}} &:= \Biggnorm{
    \biggnorm{ 2^{\nu\alpha(x)} \int_{\abs{h} \leq 1}  \bigabs{
        \Delta^M(2^{-k} h, f)(x)} \,dh }_{l^q_\nu} }_{L^p_x}
  + \norm{f}_{L^p_x},
\end{align*}
where
\begin{align*}
  \Delta^M(y,f)(x) := \sum_{k=0}^M (-1)^{M-k} \binom{M}{k}\, f(x + ky).
\end{align*}
In~\cite[Theorem~3.2]{Be03} Besov proved that
$\Fspq{\alpha(\cdot),\text{Besov}}{p}{q}(\R^n)$ can be renormed by
\begin{align*}
  \Bignorm{
    \bignorm{ 2^{\nu \alpha(x)}\, \phi_\nu \ast f(x)}_{l^q_\nu}
  }_{L^p_x} &\approx
 \norm{f}_{\Fspq{\alpha(\cdot),\text{Besov}}{p}{q}},
\end{align*}
which agrees with our definition of the norm of
$\Fspq{\alpha(\cdot)}{p}{q}$, since $p$ and $q$ are constants.
This immediately implies the following result:

\begin{theorem}
  \label{eq:besov=Fspq}
  Let $p,q \in (1,\infty)$, $\alpha \in C^{\ln}_\loc \cap L^\infty$
  and $\alpha\geq 0$. Then
  $\norm{f}_{\Fspq{\alpha(\cdot),\text{Besov}}{p}{q}}(\R^n) \approx
  \norm{f}_{\Fspq{\alpha(\cdot)}{p}{q}}(\R^n)$.
\end{theorem}

In his works, Besov also studied Besov spaces of variable
differentiability. For $p,q \in (1,\infty)$
and $\alpha \in C^{\ln}_\loc \cap L^\infty$ with $\alpha\geq 0$, he
defines
\begin{align*}
  \Bspq{\alpha(\cdot),\text{Besov}}{p}{q}(\R^n) &:= \Bigset{ f \in
    L^p_{\loc}(\R^n) \,:\,
    \norm{f}_{\Bspq{\alpha(\cdot),\text{Besov}}{p}{q}(\R^n)} <
    \infty},
  \\
  \norm{f}_{\Bspq{\alpha(\cdot),\text{Besov}}{p}{q}(\R^n)} &:=
  \Biggnorm{ \biggnorm{ \sup_{\abs{h} \leq 1}
      \bigabs{\Delta^M(2^{-\nu\alpha(x)}h, f)(x)} }_{L^p_x}
  }_{l^q_\nu} + \norm{f}_{L^p_x}.
\end{align*}

\begin{remark}
In fact, Besov gives a slightly more general definition
than this for both the Triebel--Lizorkin and the Besov spaces.
He replaces $2^{\nu\alpha(x)}$ by a sequence of functions $\beta_\nu(x)$.
The functions $\beta_\nu(x)$ are then assumed to satisfy some regularity
assumptions with respect to $\nu$ and $x$, which are very closely related to the
local $\log$-H{\"o}lder continuity of $\alpha$. Indeed, if $\beta_\nu(x) =
2^{\nu \alpha(x)}$, then his conditions on $\beta_\nu$ are precisely
that $\alpha \in C^{\ln}_\loc \cap L^\infty$.
\end{remark}

In the classical case the scale of Triebel-Lizorkin spaces and the
scale of Besov spaces agree if $p=q$. Besov showed in~\cite{Be05}
that this is also the case for his new scales of Triebel-Lizorkin
and Besov spaces, i.e.,
$\Fspq{\alpha(\cdot),\text{Besov}}{p}{p}(\R^n) =
\Bspq{\alpha(\cdot),\text{Besov}}{p}{p}(\R^n)$ for $p \in
(1,\infty)$, $\alpha \in C^{\ln}_\loc \cap L^\infty$, and
$\alpha\geq 0$. This enables us to point out a connection to another
family of spaces. By means of the symbols of pseudodifferential
operators, Leopold~\cite{Leo87} introduced Besov spaces with
variable differentiability
$\Bspq{\alpha(\cdot),\text{Leopold}}{p}{p}(\R^n)$. He further
showed that if $0<\alpha^-\leq \alpha^+ < \infty$ and $\alpha \in
C^\infty(\R^n)$, then the spaces
$\Bspq{\alpha(\cdot),\text{Leopold}}{p}{p}(\R^n)$ can be
characterized by means of finite differences. This characterization
agrees with the one that later Besov~\cite{Be99} used in the
definition of the spaces
$\Bspq{\alpha(\cdot),\text{Besov}}{p}{p}(\R^n)$. In particular,
we have $\Bspq{\alpha(\cdot),\text{Leopold}}{p}{p}(\R^n) =
\Bspq{\alpha(\cdot),\text{Besov}}{p}{p}(\R^n) =
F^{\alpha(\cdot)}_{p,p}(\R^n)$ for such~$\alpha$.

\subsection*{Other spaces}

It should be mentioned that there have recently also been some
extensions of variable integrability spaces in other directions, not
covered by the Triebel--Lizorkin scale that we introduce here.  For
instance, Harjulehto \& H{\"a}st{\"o} \cite{HH5} modified
the Lebesgue space scale on the upper end to account for the fact
that $W^{1,n}$ does not map to $L^\infty$ under the Sobolev
embedding. Similarly, in the image restoration model by Chen, Levine
and Rao mentioned above, one has the problem that the exponent $p$
takes values in the closed interval $[1,2]$, including the lower
bound, so that one is not working with reflexive spaces.  It is
well-known that the space $BV$ of functions of bounded variation is
often a better alternative than $W^{1,1}$ when studying differential
equations. Consequently, it was necessary to modify the scale
$W^{1,p(\cdot)}$ so that the lower end corresponded to $BV$. This
was done by Harjulehto, H{\"a}st{\"o} \& Latvala in \cite{HHL2}.
Schneider \cite{Sch_pp06,Sch07} has also investigated spaces of
variable smoothness, but these spaces are not included in the scale
of Leopold and Besov. Most recently, Diening, Harjulehto,
H{\"a}st{\"o}, Mizuta \& Shimomura \cite{DieHHMS_pp07} have studied
Sobolev embeddings when $p\to 1$ using Lebesgue spaces with an $L\log
L$-character on the lower end in place of $L^1$.


\section{Multiplier theorems}
\label{sect:multipliers}

Cruz-Uribe, Fiorenza, Martell and P{\'e}rez \cite[Corollary~2.1]{CFMP}
proved a very general extra\-polation theorem, which implies among other
things the following vector-valued maximal inequality, for variable
$p$ but constant $q$:

\begin{lemma}\label{CFMPvectorThm}
  Let $p \in C^{\ln}(\R^n)$ with $1 < p^- \leq p^+ < \infty$ and $1 <
  q < \infty$. Then
  \begin{align*}
  \big\| \,\| Mf_i\|_{l^q} \big\|_{p(\cdot)} \le C \big\| \,\|
  f_i\|_{l^q} \big\|_{p(\cdot)}.
  \end{align*}
\end{lemma}

It would be very nice to generalize this estimate to the variable $q$ case.
In particular, this would allow us to use classical machinery to
deal with Triebel--Lizorkin spaces. Unfortunately, it turns out that it
is not possible:
if $q$ is not constant, then the inequality
\begin{align*}
  \bignorm{ \norm{ Mf_i }_{\lqxnu} }_{\Lpdotx} \le C\, \bignorm{
    \,\norm{ f_i}_{\lqxnu} }_{\Lpdotx}
\end{align*}
does not hold, even if $p$ is constant or $p(\cdot)=q(\cdot)$. For a
concrete counter-example consider $q$ with $q|_{\Omega_j}
= q_j$, $j=0,1$, and $q_0 \not= q_1$ and a constant $p$. Set $f_k
:= a_k\, \chi_{\Omega_0}$. Then $M f_k|_{\Omega_{1}} \geq c\, a_k\,
\chi_{\Omega_{1}}$. This shows that $l^{q_0} \hookrightarrow l^{q_1}$.
The opposite embedding follows in the same way, hence, we would
conclude that $l^{q_0} \cong l^{q_1}$, which is of course false.

In lieu of a vector-valued maximal inequality,
we prove in this section estimates which take into
account that there is a clear stratification in the Triebel--Lizorkin
space, namely, a given magnitude of cube size is used in exactly one
term in the sum. Recall that $\eta_m(x)=(1+|x|)^{-m}$ and
$\eta_{\nu,m}(x) = 2^{n\nu} \eta_m(2^\nu x)$. For a measurable
set $Q$ and an integrable function $g$ we denote
\begin{align*}
M_Q g := \fint_Q |g(x)|\, dx.
\end{align*}

\begin{lemma}
  \label{lem:etavsM}
  For every $m >n$ there exists $c = c(m,n)>0$ such that
  \begin{align*}
    \eta_{\nu,m} * \abs{g}\,(x) &\leq c\, \sum_{j \geq 0} 2^{-j(m-n)}
    \sum_{Q \in \mathcal{D}_{\nu -j}} \chi_{3Q}(x)\, M_Q g
  \end{align*}
for all $\nu \geq 0$, $g \in L^1_\loc$, and $x \in \R^n$.
\end{lemma}

\begin{proof}
Fix $\nu \geq 0$, $g \in L^1_\loc$, and $x,y \in \R^n$.
If $\abs{x-y} \leq 2^{-\nu}$, then we choose
$Q\in \mathcal{D}_\nu$ which contains $x$ and $y$.
If $\abs{x-y} > 2^{-\nu}$, then we choose $j \in \N_0$ such that
  $2^{\nu-j} \leq \abs{x-y} \leq 2^{\nu-j+1}$ and let $Q \in
  \mathcal{D}_{\nu-j}$ be the cube containing~$y$.
  Note that $x \in 3Q$.
  In either case, we conclude that
    \begin{align*}
    2^{\nu n} \big(1+2^\nu\, \abs{x-y} \big)^{-m} \leq c\,
    2^{-j(m-n)} \chi_{3Q}(x)\, \frac{\chi_Q(y)}{\abs{Q}}.
  \end{align*}
Next we multiply this inequality by $|g(y)|$ and integrate
  with respect to $y$ over $\Rn$. This gives
  $\eta_{\nu,m} * \abs{g}\,(x) \leq c\, 2^{j(m-n)}\chi_{3Q}(x)\, M_Q g$,
which clearly implies the claim.
\end{proof}

For the proof of the Lemma \ref{thm:eta_old} we need the following result on the
maximal operator. It follows from Lemma~3.3 and
Corollary~3.4, \cite{DieHHMS_pp07}, since $p^+ < \infty$ in our case.

\begin{lemma}
  \label{cor:max_p_inside4}
  Let $p \in C^{\ln}(\R^n)$ with $1 \leq p^- \leq p^+ < \infty$.
  Then there exists $h \in \mathrm{weak\text{-}}L^1(\R^n) \cap
  L^\infty(\R^n)$ such that
  \begin{align*}
    M f(x)^{p(x)} &\leq
    c\, M\big(\abs{f(\cdot)}^{p(\cdot)}\big) (x) + \min\{|Q|,1\} h(x)
  \end{align*}
  for all $f \in L^{p(\cdot)}(\R^n)$ with
  $\norm{f}_{L^{p(\cdot)}(\R^n)} \leq 1$.
\end{lemma}

We are now ready for a preliminary version of Theorem~\ref{thm:eta},
containing an additional condition.

\begin{lemma}
  \label{thm:eta_old}
  Let $p, q \in C^{\ln}(\R^n)$, $1 < p^- \leq p^+ < \infty$, $1 < q^-
  \leq q^+ < \infty$, and $(p/q)^-\cdot q^- > 1$. Then there exists $m>n$
  such that
  \begin{align*}
    \Bignorm{ \bignorm{ \eta_{\nu,m} \ast f_\nu }_{\lqxnu} }_{\Lpdotx}
    &\leq c\, \Bignorm{ \bignorm{ f_\nu }_{\lqxnu} }_{\Lpdotx}
  \end{align*}
  for every sequence $\set{f_\nu}_{\nu \in \N_0}$ of
  $L^1_\loc$-functions.
\end{lemma}

\begin{proof}
  By homogeneity, it suffices to consider the case
  \begin{align*}
    \Bignorm{ \bignorm{ f_\nu }_{\lqxnu} }_{\Lpdotx} &\leq 1.
  \end{align*}
  Then, in particular,
  \begin{align}
    \label{eq:2}
    \int_{\R^n} \abs{f_\nu(x)}^{p(x)}\,dx &\leq 1
  \end{align}
  for every $\nu\ge 1$.
  Using Lemma~\ref{lem:etavsM} and Jensen's inequality (i.e.,\ the
  embedding in weighted discrete Lebesgue spaces), we estimate
  \begin{align*}
    \lefteqn{\int_{\R^n} \Bigg( \sum_{\nu \geq 0}
      \abs{\eta_{\nu,m}
        \ast f_\nu(x)}^{q(x)} \Bigg)^{\frac{p(x)}{q(x)}}\,dx}
    \hspace{5mm} &
    \\
    &\leq \int_{\R^n} \Bigg( \sum_{\nu \geq 0} \bigg( \sum_{j \geq
      0} 2^{-j(m-n)} \sum_{Q \in \mathcal{D}_{\nu -j}} \chi_{3Q}(x)\, M_Q
    f_\nu\bigg)^{q(x)} \Bigg)^{\frac{p(x)}{q(x)}}\,dx
    \\
    &\leq c\int_{\R^n} \Bigg( \sum_{\nu \geq 0} \sum_{j \geq 0}
    2^{-j(m-n)} \bigg( \sum_{Q \in \mathcal{D}_{\nu -j}} \chi_{3Q}(x)\,
    M_Q f_\nu\bigg)^{q(x)} \Bigg)^{\frac{p(x)}{q(x)}}\,dx
    \\
    &\le c \int_{\R^n} \Bigg( \sum_{\nu \geq 0} \sum_{j \geq 0}
    2^{-j(m-n)} c\,\sum_{Q \in \mathcal{D}_{\nu -j}} \chi_{3Q}(x)\,
    ( M_Q f_\nu)^{q(x)} \Bigg)^{\frac{p(x)}{q(x)}}\,dx.
  \end{align*}
  For the last inequality we used the fact that the innermost
  sum contains only a finite, uniformly bounded number of non-zero terms.

  It follows from~\eqref{eq:2} and $p(x) \geq \frac{q(x)}{q^-}$ that
$\norm{f_\nu}_{L^{\frac{q(\cdot)}{q^-}}} \leq c$.
  Thus, by Lemma~\ref{cor:max_p_inside4},
  \begin{align*}
    ( M_Q f_\nu)^{\frac{q(x)}{q^-}} &\leq c\,
    M_Q\Big(\abs{f_\nu}^{\frac{q}{q^-}} \Big) + c\,\min
    \set{\abs{Q}, 1}\, h(x)
  \end{align*}
  for all $Q \in \mathcal{D}_{\nu - j}$ and $x \in Q$.  Combining this
  with the estimates above, we get
  \begin{align*}
    \lefteqn{\int_{\R^n} \bigg( \sum_{\nu \geq 0} \abs{\eta_\nu
        \ast f_\nu(x)}^{q(x)} \bigg)^{\frac{p(x)}{q(x)}}\,dx}
    \hspace{5mm} &
    \\
    &\leq c\,\int_{\R^n} \bigg( \sum_{\nu \geq 0} \sum_{j \geq 0}
    2^{-j(m-n)} \sum_{Q \in \mathcal{D}_{\nu -j}} \chi_{3Q}(x)\, \Big[
    M_Q\Big(\abs{f_\nu}^{\frac{q}{q^-}}\Big) \Big]^{q^-}
    \bigg)^{\frac{p(x)}{q(x)}}\,dx
    \\
    &\hphantom{\leq}\quad + c\,\int_{\R^n} \bigg( \sum_{\nu \geq 0}
    \sum_{j \geq 0} 2^{-j(m-n)} \sum_{Q \in \mathcal{D}_{\nu -j}}
    \chi_{3Q}(x)\, \big( \min \set{\abs{Q}, 1}\, h(x) \big)^{q^-}
    \bigg)^{\frac{p(x)}{q(x)}}\,dx
    \\
    &=: (I) + (II).
  \end{align*}
  Now we easily estimate that
  \begin{align*}
    (I) &\leq c\,\int_{\R^n} \bigg( \sum_{\nu \geq 0} \Big[
    M\Big(\abs{f_\nu}^{\frac{q}{q^-}}\Big)(x) \Big]^{q^-} \sum_{j \geq
      0} 2^{-j(m-n)} \sum_{Q \in \mathcal{D}_{\nu -j}} \chi_{3Q}(x)\,
    \bigg)^{\frac{p(x)}{q(x)}}\,dx
    \\
    &\leq c\,\int_{\R^n} \bigg( \sum_{\nu \geq 0} \Big[
    M\Big(\abs{f_\nu}^{\frac{q}{q^-}}\Big)(x) \Big]^{q^-}
    \bigg)^{\frac{p(x)}{q(x)}}\,dx
    \\
    &= c\, \int_{\R^n} \Bignorm{ M\Big(\abs{f_\nu}^{\frac{q}{q^-}}
      \Big)(x) }_{l^{q^-}_\nu}^{\frac{p(x)}{q(x)} q^-}\,dx.
  \end{align*}
  The vector valued maximal inequality, Lemma~\ref{CFMPvectorThm},
  with $(p/q)^-\cdot q^->1$ and $q^->1$, implies that the last expression is
  bounded since
  \begin{align*}
    \int_{\R^n} \bigg( \sum_{\nu \geq 0}
    \Big(\abs{f_\nu(x)}^{\frac{q(x)}{q^-}}\Big)^{q^-}
    \bigg)^{\frac{p(x)}{q(x)}}\,dx = \int_{\R^n} \bigg( \sum_{\nu
      \geq 0} \abs{f_\nu(x)}^{q(x)} \bigg)^{\frac{p(x)}{q(x)}}\,dx
    \leq 1.
  \end{align*}
  For the estimation of $(II)$ we first note the inequality
  \begin{align*}
  \sum_{\nu \geq 0} \sum_{j \geq
      0} 2^{-j(m-n)} \sum_{Q \in \mathcal{D}_{\nu -j}} \chi_{3Q}(x)\,
    \min \set{\abs{Q}, 1}^{q^-}
    &\le
    \sum_{\nu \geq 0}
    \sum_{j \geq 0} 2^{-j(m-n)} \min \set{2^{n(j-\nu)q^-}, 1}\\
    &\le
    \sum_{j \geq 0}
    2^{-j(m-n)} \Big( j + \sum_{\nu > j} 2^{n(j-\nu)q^-} \Big)\\
    &\le
    \sum_{j \geq 0}
    2^{-j(m-n)} ( j + 1 )
    \le
    c.
  \end{align*}
  We then estimate $(II)$ as follows:
  \begin{align*}
    (II) &\leq c\,\int_{\R^n} \bigg( h(x)^{q^-} \sum_{\nu \geq 0} \sum_{j \geq
      0} 2^{-j(m-n)} \sum_{Q \in \mathcal{D}_{\nu -j}} \chi_{3Q}(x)\,
    \min \set{\abs{Q}^{q^-} , 1}\,
    \bigg)^{\frac{p(x)}{q(x)}}\,dx\\
     &\leq c\,\int_{\R^n} h(x)^{\frac{p(x)}{q(x)}q^-}\,dx.
  \end{align*}
  Since $(p/q)^-\, q^- > 1$ and $h \in \text{weak-}L^1 \cap L^\infty$,
  the last expression is bounded.
\end{proof}


Using a partitioning trick, it is possible to remove the strange condition
$(p/q)^- \cdot q^- > 1$ from the previous lemma and prove our main result regarding
multipliers:

\begin{proof}[Proof of Theorem~\ref{thm:eta}]
  Because of the uniform continuity of $p$ and $q$, we can choose a
  finite cover $\set{\Omega_i}$ of $\Rn$ with the following properties:
  \begin{enumerate}
  \item each $\Omega_i\subset \R^n$, $1\le i\le k$, is open;
  \item the sets $\Omega_i$ cover $\R^n$, i.e.,\ $\bigcup_i \Omega_i =
    \Rn$;
  \item non-contiguous sets are separated in the sense that
    $d(\Omega_i, \Omega_j)>0$ if $|i-j|>1$; and
  \item we have $(p/q)_{A_i}^- q_{A_i}^- > 1$ for $1\le i\le k$, where
    $A_i:= \bigcup_{j=i-1}^{i+1} \Omega_i$ (with the understanding that
    $\Omega_0 = \Omega_{k+1}=\emptyset$).
  \end{enumerate}
  Let us choose an integer $l$ so that $2^l\le
  \min_{|i-j|>1} 3 d(\Omega_i, \Omega_j)<2^{l + 1}$.  Since there
  are only finitely many indices, the third condition implies that
  such an $l$ exists.

  Next we split the problem and work with the domains $\Omega_i$.  In
  each of these we argue as in the previous lemma to conclude that
\begin{align*}
\begin{split}
  &\int_{\Rn} \bigg(\sum_{\nu \geq 0}
  \abs{\eta_{\nu,m} \ast f_\nu(x)}^{q(x)} \bigg)^{\frac{p(x)}{q(x)}}\,dx
\le \sum_{i=1}^k \int_{\Omega_i} \bigg(\sum_{\nu \geq 0}
  \abs{\eta_{\nu,m} \ast f_\nu(x)}^{q(x)} \bigg)^{\frac{p(x)}{q(x)}}\,dx\\
  & \quad \leq c\sum_{i=1}^k \int_{\Omega_i} \bigg( \sum_{\nu \geq 0} \sum_{j \geq
    0} 2^{-j(m-n)} \,\sum_{Q \in \mathcal{D}_{\nu -j}} \chi_{3Q}(x)\,
  (M_Q f_\nu)^{q(x)} \bigg)^{\frac{p(x)}{q(x)}}\,dx.
\end{split}
\end{align*}
{}From this we get
\begin{align*}
\begin{split}
  \int_{\Omega_i} \bigg(\sum_{\nu \geq 0}
  \abs{\eta_{\nu,m} \ast f_\nu(x)}^{q(x)} \bigg)^{\frac{p(x)}{q(x)}}\,dx
  &\leq c\int_{\Omega_i} \bigg( \sum_{\nu \geq 0}
  \sum_{j=0}^{\nu + l} 2^{-j(m-n)} \,\sum_{Q \in \mathcal{D}_{\nu
      -j}} \chi_{3Q}(x)\,
  (M_Q f_\nu)^{q(x)} \bigg)^{\frac{p(x)}{q(x)}}\,dx \\
  & \qquad+ c\int_{\Omega_i} \bigg( \sum_{\nu \geq 0} \sum_{j \geq \nu
    + l} 2^{-j(m-n)} \,M f_\nu(x)^{q(x)}
  \bigg)^{\frac{p(x)}{q(x)}}\,dx.
\end{split}
\end{align*}
The first integral on the right hand side is handled as in the previous
proof. This is possible, since the cubes in this integral are always
in $A_i$ and $(p/q)^-_{A_i} q^-_{A_i} >1$.

So it remains only to bound
\begin{align*}
  \int_{\Omega_i} \bigg( \sum_{\nu \geq 0} \sum_{j \geq \nu + l}
  2^{-j(m-n)} \,M f_\nu(x)^{q(x)} \bigg)^{\frac{p(x)}{q(x)}}\,dx & \le
  c \int_{\Omega_i} \bigg( \sum_{\nu \geq 0} 2^{-(m-n) \nu}  M
  f_\nu(x)^{q(x)} \bigg)^{\frac{p(x)}{q(x)}}\,dx.
\end{align*}
For a non-negative sequence $(x_i)$ we have
\begin{align*}
  \Big( \sum_{i\ge 0} 2^{-i(m-n)} x_i
  \Big)^r &\leq
\begin{cases}
c(r)\,\sum_{i\ge 0} 2^{-i(m-n)} x_i^r & \text{if } r \geq 1\\
 \sum_{i\ge 0} 2^{-i(m-n)r} x_i^r, & \text{if } r \leq 1.
 \end{cases}
\end{align*}
We apply this estimate for $r=\frac{p(x)}{q(x)}$ and conclude that
\begin{align*}
 \int_{\Omega_i} \bigg( \sum_{\nu \geq 0} 2^{-(m-n) \nu} \,
    M f_\nu(x)^{q(x)} \bigg)^{\frac{p(x)}{q(x)}}\,dx
 \le c \sum_{\nu \geq 0} 2^{-(m-n) \nu \min\set{1,(\frac{p}{q})^-}}
  \int_{\Omega_i}  M f_\nu(x) ^{p(x)} \,dx.
\end{align*}
The boundedness of the maximal operator implies that the integral may
be estimated by a constant, since $\int \abs{f_\nu(x)}^{p(x)} \, dx\le
1$. We are left with a geometric sum, which certainly converges.
\end{proof}


\section{Proofs of the decomposition results}
\label{sect:decomposition}

We can often take care of the variable smoothness simply by
treating it as a constant in a cube, which is what the next lemma is for.

\begin{lemma}
  \label{lem:beta_eta}
 Let $\alpha$ be as in the Standing Assumptions. There exists $d\in (n,\infty)$
such that if  $m>d$, then
  \begin{align*}
    2^{\nu\alpha(x)} \eta_{\nu,2m}(x-y) &\leq c 2^{\nu\alpha(y)} \eta_{\nu,m}(x-y)
  \end{align*}
  for all $x,y\in \Rn$.
\end{lemma}

\begin{proof}
  Choose $k \in \N_0$ as small as possible subject to the condition
  that $\abs{x-y} \leq 2^{-\nu + k}$. Then $1+ 2^{\nu} \abs{x-y}\approx
  2^k$.  We estimate that
  \begin{align*}
    \frac{\eta_{\nu,2m}(x-y)}{\eta_{\nu,m}(x-y)} &\leq c\, (1+
    2^k)^{-m} \leq c 2^{ -km}.
  \end{align*}
  On the other hand, the $\log$-H{\"o}lder continuity of $\alpha$ implies that
  \begin{align*}
    2^{\nu (\alpha(x)-\alpha(y))} \ge
    2^{-\nu c_\text{log} /\log(e+1/|x-y|)} \ge
    2^{-k c_\text{log}} \, |x-y|^{- c_\text{log} /\log(e+1/|x-y|)}
    \ge
    c\, 2^{-k c_\text{log}}.
  \end{align*}
The claim follows from these estimates provided we choose $m\ge c_\text{log}$.
\end{proof}

\begin{proof}[Proof of Theorem~\ref{thm:Sphi_bnd}]
  Let $f \in \Fspq{\alpha(\cdot)}{p(\cdot)}{q(\cdot)}$. Then we have the
  representation
  \begin{align*}
    f &= \sum_{Q \in \mathcal{D}^+} \skp{\phi_Q}{f}\, \psi_Q = \sum_{Q
      \in \mathcal{D}^+} \abs{Q}^{\frac{1}{2}} \phi_\nu * f(x_Q)\,
    \psi_Q.
  \end{align*}
Let $r \in (0,\min
  \set{p^-,q^-})$ and let $m$ be so large that
  Lemma~\ref{lem:beta_eta} applies. The functions $\phi_\nu * f$
  fulfill the requirements of Lemma~\ref{lem:est_g}, so
  \begin{align*}
    \norm{S_\phi f}_{\fspq{\alpha(\cdot)}{p(\cdot)}{q(\cdot)}}
    &= \Biggnorm{ \biggnorm{ 2^{\nu \alpha(x)} \sum_{Q \in
          \mathcal{D}_\nu} \phi_\nu * f(x_Q)\, \chi_Q}_{\lqxnu} }_{\Lpdotx}
    \\
    &\leq c\, \Bignorm{ \bignorm{2^{\nu \alpha(x)} \big(\eta_{\nu,2m} *
        \abs{\phi_\nu * f}^r\big)^{\frac{1}{r}}}_{\lqxnu} }_{\Lpdotx}
    \\
    &= c\, \Bignorm{ \bignorm{2^{\nu \alpha(x)r} \eta_{\nu,2m} *
        \abs{\phi_\nu * f}^r}_{\lqrxnu} }_{\Lprdotx}^{\frac{1}{r}}.
  \end{align*}
  By Lemma~\ref{lem:beta_eta} and Theorem~\ref{thm:eta}, we further conclude
that
  \begin{align*}
    \norm{S_\phi f}_{\fspq{\alpha(\cdot)}{p(\cdot)}{q(\cdot)}}
&\leq
    c\, \Bignorm{ \bignorm{\eta_{\nu,m} * \big( 2^{\nu \alpha(\cdot)}
        \abs{\phi_\nu * f} \big)^r}_{\lqrxnu}
    }_{\Lprdotx}^{\frac{1}{r}}\\
    &\leq c\, \Bignorm{ \bignorm{2^{\nu \alpha(x) r} \abs{\phi_\nu * f}^r
      }_{\lqrxnu} }_{\Lprdotx}^{\frac{1}{r}}
    \\
    &= c\, \Bignorm{ \bignorm{2^{\nu \alpha(x)} \phi_\nu * f
      }_{\lqxnu} }_{\Lpdotx}.
  \end{align*}
  This proves the theorem.
\end{proof}

In order to prove Theorem~\ref{thm:Sphi_inv} we need to split our
domain into several parts. The following lemma will be applied to
each part. For the statement we need Triebel--Lizorkin spaces
defined in domains of $\Rn$. These are achieved simply by replacing
$L^{p(\cdot)}(\Rn)$ by $L^{p(\cdot)}(\Omega)$ in the definitions of
$\Fspq{\alpha(\cdot)}{p(\cdot)}{q(\cdot)}$ and
$\fspq{\alpha(\cdot)}{p(\cdot)}{q(\cdot)}$:
\begin{align*}
  \norm{f}_{\Fspq{\alpha(\cdot)}{p(\cdot)}{q(\cdot)}(\Omega)} &:=
  \Bignorm{ \bignorm{ 2^{\nu \alpha(x)}\, \phi_\nu \ast
      f(x)}_{\lqxnu} }_{\Lpdotx(\Omega)}
\end{align*}
and
\begin{align*}
  \bignorm{\set{s_Q}_{Q}}_{\fspq{\alpha(\cdot)}{p(\cdot)}{q(\cdot)}(\Omega)}
  &:= \Biggnorm{ \biggnorm{ 2^{\nu\alpha(x)} \sum_{Q \in
        \mathcal{D}_\nu} \abs{s_Q}\, \abs{Q}^{-\frac{1}{2}}\, \chi_Q
    }_{\lqxnu} }_{\Lpdotx(\Omega)}.
\end{align*}

\begin{lemma} \label{lem:Sphi_inv} Let $p$, $q$, and $\alpha$ be as in
  the Standing Assumptions and define functions $J=n/\min\{1,p,q\}$
  and $N=J-n-\alpha$. Let $\Omega$ be a cube or the complement of a
  finite collection of cubes and suppose that $\{m_Q\}_Q$, $Q\subset
  \Omega$, is a family of
$(J^+ - n - \alpha^- + \epsilon,\alpha^+ +1+ \epsilon)$-smooth
molecules, for some $\epsilon > 0$.  Then
  \begin{align*}
    \norm{f}_{\Fspq{\alpha(\cdot)}{p(\cdot)}{q(\cdot)}(\Omega)} \leq c\,
    \norm{\set{s_Q}_Q}_{\fspq{\alpha(\cdot)}{p(\cdot)}{q(\cdot)}(\Omega)},
    \quad\text{where}\quad
    f = \sum_{\nu \geq 0}
    \sum_{\stackrel{Q \in \mathcal{D}_\nu}{Q\subset\Omega}} s_Q m_Q
  \end{align*}
  and $c>0$ is independent of $\set{s_Q}_Q$ and $\set{m_Q}_Q$\,.
\end{lemma}

\begin{proof}
  Let $2m$ be sufficiently large, i.e., larger than $M$ (from the definition
of molecules.  Choose $r \in (0,\min\set{1,p^-,q^-})$, $\epsilon > 0$,
  $k_1 \geq \alpha^+ + 2\epsilon$ and
  $k_2 \geq \frac{n}{r} - n - \alpha^- + 2\epsilon$
so that $\{m_Q\}$ are $(k_2,k_1+1,2m)$-smooth molecules.
 Define
  $k(\nu,\mu) := k_1\, (\nu - \mu)_+ +k_2\, (\mu-\nu)_+$ and
  $\widetilde{s}_{Q_\mu} := s_{Q_\mu}\, \abs{Q_\mu}^{-1/2}$.

  Next we apply Lemma~\ref{lem:van_moments} twice:
  with $g=\phi_\nu$, $h(x)=m_{Q_\mu}(x-x_{Q_\mu})$ and $k=\lfloor k_2 \rfloor+1$
  if $\mu\ge \nu$, and
  $g(x)=m_{Q_\mu}(x-x_{Q_\mu})$, $h=\phi_\nu$ and $k=\lfloor k_1 \rfloor+1$
  otherwise. This and Lemma~\ref{lem:conv_eta_chi} give
  \begin{align*}
    \abs{\phi_\nu \ast m_{Q_\mu}(x)}
    &\le c 2^{-k(\nu,\mu)} \abs{Q_\mu}^{1/2} \eta_{\nu,2m} \ast \eta_{\mu,2m}(x+x_{Q_\mu}) \\
    &\approx c 2^{-k(\nu,\mu)} \abs{Q_\mu}^{-1/2} (\eta_{\nu,2m} \ast
    \eta_{\mu,2m} \ast \chi_{Q_\mu})(x).
  \end{align*}
  Thus, we have
  \begin{align*}
    \norm{f}_{\Fspq{\alpha(\cdot)}{p(\cdot)}{q(\cdot)}(\Omega)}
    &= \Biggnorm{ \biggnorm{ \sum_{\mu \geq 0}
        \smash{\sum_{\substack{Q_\mu
            \in \mathcal{D}_\mu \\ Q_\mu \subset \Omega}}}
        2^{\nu\alpha(x)} |s_{Q_\mu}|\, \phi_\nu \ast
        m_{Q_\mu}}_{\lqxnu} }_{\Lpdotx(\Omega)}
    \\
    &\leq \Biggnorm{ \biggnorm{ \sum_{\mu \geq0} \smash{\sum_{Q_\mu \in
          \mathcal{D}_\mu}} \abs{\widetilde{s}_{Q_\mu}}\,
        2^{\nu\alpha(x)-k(\nu,\mu)} \eta_{\nu,2m} \ast \eta_{\mu,2m}
        \ast \chi_{Q_\mu}}_{\lqxnu} }_{\Lpdotx(\Omega)}
    \\
    &= \Biggnorm{ \biggnorm{ \bigg( \sum_{\mu \geq 0} \sum_{Q_\mu \in
          \mathcal{D}_\mu} \abs{\widetilde{s}_{Q_\mu}}\,
        2^{\nu\alpha(x)-k(\nu,\mu)} \eta_{\nu,2m} \ast \eta_{\mu,2m}
        \ast \chi_{Q_\mu}\bigg)^r}_{\lqrxnu}
    }_{\Lprdotx(\Omega)}^{\frac{1}{r}}.
  \end{align*}

  Next we use the embedding $l^r\hookrightarrow l^1$ and obtain the estimate on the
  term inside of the two norms above as follows
  \begin{align*}
    & \bigg( \sum_{\mu \geq 0} \sum_{Q_\mu \in \mathcal{D}_\mu}
       \abs{\widetilde{s}_{Q_\mu}} \, 2^{\nu\alpha(x)-k(\nu,\mu)}
      \eta_{\nu,2m} \ast \eta_{\mu,2m} \ast  \chi_{Q_\mu}\bigg)^r
    \\
    &\qquad \leq \sum_{\mu \geq 0} \sum_{Q_\mu \in \mathcal{D}_\mu}
     \abs{\widetilde{s}_{Q_\mu}}^r 2^{\nu\alpha(x)r- k(\nu,\mu) r
      } (\eta_{\nu,2m} * \eta_{\mu,2m} *
    \chi_{Q_\mu})^r.
  \end{align*}
  By Lemma~\ref{lem:etapower} we conclude that
  \begin{equation}
    \label{muNuAssumption}
    \begin{split}
      & 2^{\nu\alpha(x)r - k(\nu, \mu) r} (\eta_{\nu,2m} * \eta_{\mu,2m} *
      \chi_{Q_\mu})^r
      \\
      & \qquad\leq c\, 2^{\nu\alpha(x)r -k_1 r \,(\nu - \mu)_+ -
        k_2r(\mu-\nu)_+ +n(1-r) (\nu -
          \mu)_+} \eta_{\nu,2mr} * \eta_{\mu,2mr} * \chi_{Q_\mu}
      \\
      & \qquad\leq c\,2^{\mu\alpha(x)r -2\epsilon\abs{\nu - \mu}}
      \eta_{\nu,2mr} * \eta_{\mu,2mr} * \chi_{Q_\mu},
    \end{split}
  \end{equation}
  where, in the second step, we used the assumptions on $k_1$ and
  $k_2$.  We use this with our previous estimate to get
  \begin{align*}
    \norm{f}_{\Fspq{\alpha(\cdot)}{p(\cdot)}{q(\cdot)}(\Omega)} &\leq
    \Biggnorm{ \biggnorm{ \sum_{\mu \geq 0} \sum_{Q_\mu \in
          \mathcal{D}_\mu} \abs{\widetilde{s}_{Q_\mu}}^r\, \,
        2^{\mu\alpha(x)r-2\epsilon\abs{\nu - \mu}r}\, \eta_{\nu,2mr} *
        \eta_{\mu,2mr} * \chi_{Q_\mu}}_{\lqrxnu}
    }_{\Lprdotx(\Omega)}^{\frac{1}{r}}.
  \end{align*}

  We apply Lemma~\ref{lem:beta_eta} and
  Theorem~\ref{thm:eta} to conclude that
  \begin{align*}
    \norm{f}_{\Fspq{\alpha(\cdot)}{p(\cdot)}{q(\cdot)}(\Omega)} &\leq
    \Biggnorm{ \biggnorm{ \eta_{\nu,mr} * \bigg( \sum_{\mu \geq 0}
        \sum_{Q_\mu \in \mathcal{D}_\mu} \abs{\widetilde{s}_{Q_\mu}}^r
        2^{\mu\alpha(\cdot) r-2\epsilon\abs{\nu - \mu}r}\, \eta_{\mu,2mr}
        \ast \chi_{Q_\mu} \bigg) }_{\lqrxnu}
    }_{\Lprdotx(\Omega)}^{\frac{1}{r}}
    \\
    &\leq \Biggnorm{ \biggnorm{ \sum_{\mu \geq 0} \sum_{Q_\mu \in
          \mathcal{D}_\mu} \abs{\widetilde{s}_{Q_\mu}}^r\, 2^{\mu
          \alpha(x) r-2\epsilon\abs{\nu - \mu}r}\, \eta_{\mu,2mr} *
        \chi_{Q_\mu} }_{\lqrxnu} }_{\Lprdotx(\Omega)}^{\frac{1}{r}}.
  \end{align*}
  We estimate the inner part (which depends on $x$) pointwise as follows:
  \begin{align*}
    \lefteqn{\biggnorm{ \sum_{\mu \geq 0} \sum_{Q_\mu \in
          \mathcal{D}_\mu} \abs{\widetilde{s}_{Q_\mu}}^r\,
        2^{\mu\alpha(x) r-2\epsilon\abs{\nu - \mu}r}\, \eta_{\mu,2mr}
        * \chi_{Q_\mu}}_{\lqrxnu}^{\frac{q(x)}{r}} } \hspace{5mm} &
    \\
    &= \sum_{\nu \geq 0} \biggabs{ \sum_{\mu \geq 0} \sum_{Q_\mu \in
        \mathcal{D}_\mu} \abs{\widetilde{s}_{Q_\mu}}^r\,
      2^{\mu\alpha(x) r-2\epsilon\abs{\nu - \mu}r}\, \eta_{\mu,2mr} *
      \chi_{Q_\mu} }^{\frac{q(x)}{r}}
    \\
    &\leq c\, \sum_{\nu \geq 0} \sum_{\mu \geq 0} 2^{-
      \epsilon\abs{\nu - \mu}r}\, \biggabs{ \sum_{Q_\mu \in
        \mathcal{D}_\mu} 2^{\mu\alpha(x) r}
      \abs{\widetilde{s}_{Q_\mu}}^r\, \eta_{\mu,2mr} * \chi_{Q_\mu}
    }^{\frac{q(x)}{r}},
 \end{align*}
 where, for the inequality, we used H{\"o}lder's inequality
in the space with geometrically decaying weight, as in the proof
of Theorem~\ref{thm:eta}. Now the only part which depends on $\nu$ is a
geometric sum, which we estimate by a constant. Next we change
the power $\alpha(x)$ to $\alpha(y)$ by Lemma~\ref{lem:beta_eta}:
 \begin{align*}
   \lefteqn{\biggnorm{ \sum_{\mu \geq 0} \sum_{Q_\mu \in
         \mathcal{D}_\mu} \abs{\widetilde{s}_{Q_\mu}}^r\,
       2^{\mu\alpha(x) r-2\epsilon\abs{\nu - \mu}r}\, \eta_{\mu,2mr} *
       \chi_{Q_\mu}}_{\lqrxnu}^{\frac{q(x)}{r}} } \hspace{5mm} &
   \\
   &\leq c\, \sum_{\mu \geq 0} \biggabs{ \sum_{Q_\mu \in
       \mathcal{D}_\mu} 2^{\mu\alpha(x) r}
     \abs{\widetilde{s}_{Q_\mu}}^r\, \eta_{\mu,2mr} * \chi_{Q_\mu}
   }^{\frac{q(x)}{r}}
   \\
   &\leq c\, \biggnorm{ \eta_{\mu,m} \ast \bigg(\sum_{Q_\mu \in
       \mathcal{D}_\mu} 2^{\mu\alpha(\cdot) r}
     \abs{\widetilde{s}_{Q_\mu}}^r\, \chi_{Q_\mu} \bigg)
   }_{\lqrxmu}^{\frac{q(x)}{r}}.
  \end{align*}
  Hence, we have shown that
  \begin{align*}
    \norm{f}_{\Fspq{\alpha(\cdot)}{p(\cdot)}{q(\cdot)}(\Omega)} &\leq c\,
    \Biggnorm{ \biggnorm{ \eta_{\mu,m} \ast \bigg(\sum_{Q \in
          \mathcal{D}_\mu} 2^{\mu\alpha(\cdot) r}
        \abs{\widetilde{s}_Q}^r\, \chi_Q \bigg) }_{\lqrxmu}
    }_{\Lprdotx(\Omega)}^{\frac{1}{r}}.
  \end{align*}
  Therefore, by Theorem~\ref{thm:eta}, we conclude that
  \begin{align*}
    \norm{f}_{\Fspq{\alpha(\cdot)}{p(\cdot)}{q}(\Omega)} &\leq c\, \Biggnorm{
      \biggnorm{ \sum_{Q \in \mathcal{D}_\mu} 2^{\mu\alpha(x) r}
        \abs{\widetilde{s}_Q}^r\, \chi_Q }_{\lqrxmu}
    }_{\Lprdotx(\Omega)}^{\frac{1}{r}}
    \\
    &= c\, \Biggnorm{ \biggnorm{ \sum_{Q \in \mathcal{D}_\mu}
        2^{\mu\alpha(x)} \abs{s_Q} \, \abs{Q}^{-\frac{1}{2}}\, \chi_Q
      }_{\lqxmu} }_{\Lpdotx(\Omega)} =
    \norm{\set{s_Q}_Q}_{\fspq{\alpha(\cdot)}{p(\cdot)}{q(\cdot)}(\Omega)},
  \end{align*}
  where we used that the sum consists of a single non-zero
  term.
\end{proof}

\begin{proof}[Proof of Theorem~\ref{thm:Sphi_inv}]
We will reduce the claim to the previous lemma.

By assumption there exists $\epsilon>0$ so that
the molecules $m_Q$ are $(N + 4\epsilon,\alpha + 1 + 3\epsilon)$-smooth.
By the uniform continuity of $p$, $q$ and $\alpha$,
we may choose $\mu_0\ge 0$ so that
$N_{Q}^- > J_{Q}^+ - \alpha_Q^- - n + \epsilon$ and
$\alpha_{Q}^- > \alpha_{Q}^+ - \epsilon$ for every
dyadic cube $Q$ of level $\mu_0$.
Note that if $Q_0$ is a dyadic cube of level $\mu_0$ and
$Q\subset Q_0$ is another dyadic cube, then
\[
N_{Q}^- \ge N_{Q_0}^- > J_{Q_0}^+ - \alpha_{Q_0}^- - n - \epsilon \ge
J_{Q}^+ - \alpha_Q^- - n  - \epsilon,
\]
similarly for $\alpha$. Thus we conclude that
$m_Q$ is a $(J_{Q}^+ - \alpha_Q^- - n  + 3\epsilon,\alpha_Q^+ + 1 + 2\epsilon)$-smooth
when $Q$ is of level at most $\mu_0$.

Since $p$, $q$ and $\alpha$ have a limit at infinity, we conclude that
$N_{\Rn\setminus K}^- > J_{\Rn\setminus K}^+ - \alpha_{\Rn\setminus K}^- - n +  - \epsilon$ and
$\alpha_{\Rn\setminus K}^- > \alpha_{\Rn\setminus K}^+ - \epsilon$
for some compact set $K\subset\Rn$.
We denote by $\Omega_i$, $i=1,\ldots,M$, those dyadic cubes
of level $\mu_0$ which intersect $K$, and define
$\Omega_0=\Rn \setminus \bigcup_{i=1}^M \Omega_i$.

For every integer $i\in [0,M]$ choose $r_i \in
(0,\min\set{1,p_{\Omega_i}^-,q_{\Omega_i}^-})$ so that
$\frac{n}{r_i} < J_{Q}^+ +\epsilon$, and set
$k_i := \frac{n}{r_i}  - n - \alpha_{\Omega_i}^- + 2\epsilon$
and $K_i = \alpha_{\Omega_i}^+ + 2\epsilon$.
Then $m_Q$ is a $(k_i, K_i + 1 )$-smooth molecule
when $Q$ is of level at most $\mu_0$. Define
$k_i(\nu,\mu) := K_i\, (\nu - \mu)_+ +k_i\,(\mu-\nu)_+$ and
$\widetilde{s}_{Q_\mu} := s_{Q_\mu}\, \abs{Q_\mu}^{-1/2}$.
Finally, let $r \in (0,\min\set{1,p^-,q^-})$.

Note that the constants $k_i$ and $K_i$ have been chosen so
that in each set $\Omega_i$ we may argue as in the previous
lemma. Thus we get
  \begin{align*}
    \abs{\phi_\nu \ast m_{Q_\mu}(x)}
    &\le c 2^{-k(\nu,\mu)} \abs{Q_\mu}^{-1/2} (\eta_{\nu,2m} \ast
    \eta_{\mu,2m} \ast \chi_{Q_\mu})(x).
  \end{align*}
From this we conclude that
  \begin{align*}
    \begin{split}
      \norm{f}_{\Fspq{\alpha(\cdot)}{p(\cdot)}{q(\cdot)}}
      & \le
      \Bignorm{ \bignorm{2^{\nu \alpha} \phi_\nu \ast f}_{\lqxnu}
      }_{\Lpdotx}
      \\
      &\leq \Biggnorm{ \biggnorm{ \sum_{\mu = 0}^{\mu_0-1} \sum_{Q_\mu \in
            \mathcal{D}_\mu} \abs{\widetilde{s}_{Q_\mu}}\,
          2^{\nu\alpha(x)-k(\nu,\mu)} \eta_{\nu,2m} \ast \eta_{\mu,2m} \ast
          \chi_{Q_\mu}}_{\lqxnu} }_{\Lpdotx}
      \\
      & \qquad + \sum_{i=0}^M
      \Biggnorm{ \biggnorm{ \sum_{\mu \geq \mu_0-1} \sum_{Q_\mu \in
            \mathcal{D}_\mu} \abs{\widetilde{s}_{Q_\mu}}\,
          2^{\nu\alpha(x)-k(\nu,\mu)} \eta_{\nu,2m} \ast \eta_{\mu,2m} \ast
          \chi_{Q_\mu}}_{\lqxnu} }_{\Lpdotx(Q_i)}.
    \end{split}
  \end{align*}
  By the previous lemma, each term in the last sum is dominated by
  $\norm{\set{s_Q}_Q}_{\fspq{\alpha(\cdot)}{p(\cdot)}{q(\cdot)}}$, so
  we conclude that
  \begin{align*}
    \norm{f}_{\Fspq{\alpha(\cdot)}{p(\cdot)}{q(\cdot)}}
    &\leq \Biggnorm{ \biggnorm{ \sum_{\mu = 0}^{\mu_0-1} \sum_{Q_\mu \in
          \mathcal{D}_\mu} \abs{\widetilde{s}_{Q_\mu}}\,
        2^{\nu\alpha(x)-k(\nu,\mu)} \eta_{\nu,2m} \ast \eta_{\mu,2m} \ast
        \chi_{Q_\mu}}_{\lqxnu} }_{\Lpdotx} \\
    &\qquad + c (M+1) \norm{\set{s_Q}_Q}_{\fspq{\alpha(\cdot)}{p(\cdot)}{q(\cdot)}}.
  \end{align*}

  It remains only to take care of the first term on the right hand
  side.  An analysis of the proof of the previous lemma shows that the
  only part where the assumption on the smoothness of the molecules
  was needed was in the estimate \eqref{muNuAssumption}.  In the
  current case we get instead
  \begin{align*}
    & 2^{\nu\alpha(x)r -r k(\nu, \mu)}
    (\eta_{\nu,2mr} * \eta_{\mu,2mr} *
    \chi_{Q_\mu})^r
    \\
    & \qquad\leq c\,2^{\mu\alpha(x)r -2\epsilon\abs{\nu - \mu} +
      n(1-r)_+ (\mu-\nu)_+} \eta_{\nu,2mr} * \eta_{\mu,2mr} *
    \chi_{Q_\mu},
  \end{align*}
  since we have no control of $k_2$. However, since $\mu\le \mu_0$ and
  $\nu\ge 0$, the extra term satisfies $2^{n(1-r)_+ (\mu-\nu)_+} \le
  2^{n(1-r)_+\mu_0}$, so it is just a constant. After this modification
  the rest of the proof of Lemma~\ref{lem:Sphi_inv} takes care of the
  first term.
\end{proof}

\begin{proof}[Proof of Theorem~\ref{thm:atom_Fspxq}]
Define constants $K = n/\min\{1,p^-,q^-\} - n + \epsilon$ and $L = \alpha^+ + 1 + \epsilon$.
We construct $(K,L)$-smooth atoms $\{a_Q\}_{Q\in \mathcal{D}^+}$
exactly as on p.~132 of \cite{FraJa}.
Note that we may use the constant indices construction,
since the constants $K$ and $L$ give sufficient
smoothness at every point. These atoms are also atoms
for the space $\Fspq{\alpha(\cdot)}{p(\cdot)}{q(\cdot)}$.

Let $f\in \Fspq{\alpha(\cdot)}{p(\cdot)}{q(\cdot)}$.
With functions as in Definition~\ref{phiDef}, we represent
$f$ as $\displaystyle f = \sum_{Q\in \mathcal{D}^+} t_Q \phi_Q$,
where $t_Q = \langle f,\psi_Q\rangle$. Next, we define
\[
(t_r^*)_{Q_{\nu k}} = \left(\sum_{P\in \mathcal{D}_\nu}
\frac{|t_P|^r}{(1+2^{\nu}|x_P-x_Q|)^m} \right)^{1/r},
\]
for $Q = Q_{\nu k}$, $\nu \in \mathbb{N}^0$ and $k \in \mathbb{Z}^n$.
For there numbers $(t_r^*)_Q$ we know that $f = \sum_Q (t_r^*)_Q a_Q$
where $\{a_Q\}_Q$ are atoms (molecules with support in $3Q$),
by the construction of \cite{FraJa}. (Technically, the atoms from
the construction of \cite{FraJa} satisfy our inequalities for molecules
only up to a constant (independent of the cube and scale). We will
ignore this detail.)

For $\nu \in \mathbb{N}_0$ define $T_\nu := \sum_{Q\in \mathcal{D}_\nu}
t_Q\, \chi_Q$.
The definition of $t_r^*$ is a discrete convolution of $T_\nu$ with
$\eta_{\nu,m}$. Changing to the continuous version, we see that
$(t_r^*)_{Q_{\nu k}} \approx \big(\eta_{\nu,M} \ast (|T_\nu|^r) (x)\big)^{1/r}$
for $x\in Q_{\nu k}$. By this point-wise estimate we conclude that
\begin{align*}
\| t_r^* \|_{\fspq{\alpha(\cdot)}{p(\cdot)}{q(\cdot)}}
&=
\bigg\| \Big \|
\Big\{ 2^{\nu \alpha(x)} \sum_{Q\in \mathcal{D}_\nu}
|Q|^{-\frac12}\, (t_r^*)_Q\, \chi_Q\Big\}_\nu
\Big\|_{l^{q(x)}_\nu}\bigg\|_{L^{p(\cdot)}_x} \\
&\approx
\bigg \| \Big\|
\big\{ 2^{\nu \alpha(x)+\nu/2} \eta_{\nu,M}\ast (|T_\nu|^r)\Big\}_\nu
\big\|_{l^\frac{q(x)}r_\nu} \bigg\|_{L^\frac{p(\cdot)}r_x}^\frac1r.
\end{align*}
Next we use Lemma~\ref{lem:beta_eta} and Theorem~\ref{thm:eta}
to conclude that
\begin{align*}
\bigg \| \big\| &
\big\{ 2^{\nu \alpha(x)+\nu/2} \eta_{\nu,M}\ast (|T_\nu|^r)\big\}_\nu
\big\|_{l^\frac{q(x)}r_\nu} \bigg\|_{L^\frac{p(\cdot)}r_x}^\frac1r \\
&\le c\, \bigg \| \Big\| \big\{ 2^{\nu\alpha(x)+\nu/2}
T_\nu\big\}_\nu \Big\|_{l^{q(x)}_\nu} \bigg\|_{L^{p(\cdot)}_x} =
\bigg \| \Big\| \Big\{ 2^{\nu \alpha(x)} \sum_{Q\in \mathcal{D}_\nu}
|Q|^{-\frac12}\, t_Q\, \chi_Q\Big\}_\nu
\Big\|_{l^{q(x)}_\nu}\bigg\|_{L^{p(\cdot)}_x}.
\end{align*}
Since $f= \sum_{Q\in \mathcal{D}^+} t_Q \phi_Q$, Theorem~\ref{thm:Sphi_bnd}
implies that this is bounded by a constant times
$\| f \|_{\Fspq{\alpha(\cdot)}{p(\cdot)}{q(\cdot)}}$.

This completes one direction. The other direction,
$$
\| f \|_{\Fspq{\alpha(\cdot)}{p(\cdot)}{q(\cdot)}} \le c\,
\| \{s_Q\}_Q \|_{\fspq{\alpha(\cdot)}{p(\cdot)}{q(\cdot)}},
$$
follows from Theorem \ref{thm:Sphi_inv},
since every family of atoms is in particular a family of molecules.
\end{proof}

We next consider a general embedding lemma. The local classical
scale of Triebel--Lizorkin spaces is increasing in the primary index
$p$ and decreasing in the secondary index $q$. This is a direct
consequence of the corresponding properties of $L^p$ and $l^q$. In
the variable exponent setting we have the following global result
provided we assume that $p$ stays constant at infinity:

\begin{proposition}
  \label{pro:embedding}
  Let $p_j$, $q_j$, and $\alpha_j$ be as in the Standing Assumptions, $j=0,1$.
  \begin{enumerate}[label={\rm (\alph{*})}]
  \item \label{item:embedding1} If $p_0 \geq p_1$ and $(p_0)_\infty =
    (p_1)_\infty$, then $L^{p_0(\cdot)} \hookrightarrow
    L^{p_1(\cdot)}$.
  \item \label{item:embedding2} If $\alpha_0 \geq \alpha_1$, $p_0 \geq p_1$,
    $(p_0)_\infty = (p_1)_\infty$, and $q_0 \leq q_1$, then
    $\Fspq{\alpha_0(\cdot)}{p_0(\cdot)}{q_0(\cdot)} \hookrightarrow
    \Fspq{\alpha_1(\cdot)}{p_1(\cdot)}{q_1(\cdot)}$.
  \end{enumerate}
\end{proposition}

\begin{proof}
  In Lemma~2.2 of \cite{D2} it is shown that $L^{p_0(\cdot)}(\R^n) \hookrightarrow
  L^{p_1(\cdot)}(\R^n)$ if and only if $p_0 \geq p_1$ almost
  everywhere and $1 \in L^{r(\cdot)}(\R^n)$, where $\frac{1}{r(x)}
  := \frac{1}{p_1(x)} - \frac{1}{p_0(x)}$. Note that $r(x) = \infty$
  if $p_1(x) = p_0(x)$. The condition $1 \in L^{r(\cdot)}(\R^n)$
  means in this context (since $r$ is usually unbounded) that
  $\lim_{\lambda \searrow 0} \varrho_{r(\cdot)}(\lambda) =0$, where we
  use the convention that $\lambda^{r(x)} = 0$ if $r(x)=\infty$ and
  $\lambda \in [0,1)$. Due to the assumptions on $p_0$ and $p_1$, we
  have $\frac{1}{r} \in C^{\ln}$, $\frac{1}{r} \geq 0$, and
  $\frac{1}{r_\infty} = 0$. In particular, $\abs{\frac{1}{r(x)}}\leq
  \frac{A}{\ln(e+\abs{x})}$ for some $A>0$ and all $x \in \R^n$.
  Thus,
  \begin{align*}
    \varrho_{r(\cdot)}(\exp(-2nA)) &= \int_{\R^n} \exp \Bigg(
    \frac{-2nA}{\abs{\frac{1}{r(x)}}} \Bigg) \,dx \leq \int_{\R^n}
    (e + \abs{x})^{-2n}\,dx < \infty.
  \end{align*}
  The convexity of $\varrho_{r(\cdot)}$ implies that
  $\varrho_{r(\cdot)}(\lambda\, \exp(-2nA)) \to 0$ as $\lambda
  \searrow 0$ and \ref{item:embedding1} follows.

  For \ref{item:embedding2} we argue as follows.
Since $\alpha_0 \geq \alpha_1$, we have
  $2^{\nu \alpha_0(x)} \leq 2^{\nu \alpha_1(x)}$ for all $\nu \geq 0$
  and all $x \in \R^n$.  Moreover, $q_0 \leq q_1$ implies
  $\norm{\cdot}_{l^{q_1}} \leq \norm{\cdot}_{l^{q_0}}$
  and~\ref{item:embedding1} implies $L^{p_0(\cdot)}(\R^n)
  \hookrightarrow L^{p_1(\cdot)}(\R^n)$. Now, the claim follows
  immediately from the definitions of the norms of
  $\Fspq{\alpha_0(\cdot)}{p_0(\cdot)}{q_0(\cdot)}$ and
  $\Fspq{\alpha_1(\cdot)}{p_1(\cdot)}{q_1(\cdot)}$.
\end{proof}

With the help of this embedding result we can prove
the density of smooth functions.

\begin{proof}[Proof of Corollary~\ref{cor:density}]
Choose $K$ so large that $\Fspq{K}{p^+}{2} \hookrightarrow
\Fspq{\alpha^+}{p^+}{1}$. This is possible by classical,
fixed exponent, embedding results.

Let $f\in \Fspq{\alpha(\cdot)}{p(\cdot)}{q(\cdot)}$ and choose
smooth atoms $a_Q\in C^k(\Rn)$
so that $f=\sum_{Q \in \mathcal{D}^+} t_Q a_Q$
in $\mathcal{S}'$. Define
\begin{align*}
f_m = \sum_{\nu=0}^m
\sum_{Q \in \mathcal{D_\nu}, |x_Q|<m} t_Q a_Q.
\end{align*}
Then clearly $f_m \in C_0^K$ and $f_m \to f$ in
$\Fspq{\alpha(\cdot)}{p(\cdot)}{q(\cdot)}$.

We can chose a sequence of functions $\phi_{m,k}\in C^\infty_0$
so that $\| f_m -\phi_{m,k}\|_{W^{K,p^+}} \to 0$ as $k\to\infty$ and
the support of $\phi_{k,m}$ is lies in the ball $B(0,r_m)$.
By the choice of $K$ we conclude that
\[
\| f_m -\phi_{m,k}\|_{\Fspq{\alpha^+}{p^+}{1}}\le
c\| f_m -\phi_{m,k}\|_{\Fspq{K}{p^+}{2}}= c \| f_m -\phi_{m,k}\|_{W^{K,p^+}}.
\]
By Proposition~\ref{pro:embedding} we conclude that
\[
\| f_m -\phi_{m,k}\|_{\Fspq{\alpha(\cdot)}{p(\cdot)}{q(\cdot)}}
\le c \| f_m -\phi_{m,k}\|_{\Fspq{\alpha^+}{p^+}{1}}.
\]
Note that the assumption $(p_0)_\infty = (p_1)_\infty$
of the proposition is irrelevant, since our functions have bounded support.
Combining these inequalities yields that
$\phi_{m,k}\to f_m$ in $\Fspq{\alpha(\cdot)}{p(\cdot)}{q(\cdot)}$,
hence we may chose a sequence $k_m$ so that $\phi_{m,k_m} \to f$ in
$\Fspq{\alpha(\cdot)}{p(\cdot)}{q(\cdot)}$, as required.
\end{proof}

\begin{remark}
Note that we used density of smooth functions in the proof of the
equality $\Fspq{k}{p(\cdot)}{2} \cong W^{k,p(\cdot)}$.
However, in the proof of the previous corollary we needed
this result only for constant exponent:
$\Fspq{k}{p}{2} \cong W^{k,p}$. Therefore, the argument is
not circular.
\end{remark}


\section{Traces}
\label{sect:traces}

In this section we deal with trace theorems for Triebel--Lizorkin spaces.
We write $\mathcal{D}^n$ and $\mathcal{D}_\nu^n$ for the families of
dyadic cubes in $\mathcal{D}^+$
when we want to emphasize the dimension of the underlying space.
The idea of the proof of the main trace theorem is to use the
localization afforded by the atomic decomposition, and express a
function as a sum of only those atoms with support intersecting the hyperplane
$\R^{n-1}\subset\Rn$. In the classical case, this approach is due to
Frazier and Jawerth \cite{FJ1}.

There have been other approaches to deal with traces and extension
operators using wavelet decomposition instead of atomic
decomposition, which utilizes compactly supported Daubechies
wavelets, and thus, conveniently gives trace theorems (see, e.g.,
\cite{FraR_pp07blms}). However, for that one would need to define
and establish properties of almost diagonal operators and  almost
diagonal matrices for the $\Fspq{\alpha(\cdot)}{p(\cdot)}{q(\cdot)}$
and $\fspq{\alpha(\cdot)}{p(\cdot)}{q(\cdot)}$ spaces. In the
interest of brevity we leave this for future research.

The following lemma shows that it does not matter much for the norm
if we shift around the mass a bit in the sequence space.

\begin{lemma}
\label{lem:E_Q}
  Let $p$, $q$, and $\alpha$ be as in the Standing Assumptions,
  $\epsilon>0$, and let $\set{E_Q}_Q$ be a collection of sets with
  $E_Q \subset 3Q$ and $\abs{E_Q} \geq \epsilon\, \abs{Q}$. Then
  \begin{align*}
    \bignorm{ \set{s_Q}_Q}_{\fspq{\alpha(\cdot)}{p(\cdot)}{q(\cdot)}}
    &\approx \Biggnorm{ \biggnorm{ 2^{\nu\alpha(x)} \sum_{Q \in
          \mathcal{D}_\nu} \abs{s_Q}\, \abs{Q}^{-\frac{1}{2}}\,
        \chi_{E_Q} }_{\lqxnu} }_{\Lpdotx}
  \end{align*}
for all $\set{s_Q}_Q \in \fspq{\alpha(\cdot)}{p(\cdot)}{q(\cdot)}$.
\end{lemma}

\begin{proof}
  We start by proving the inequality ``$\leq$''.
  Let $r \in (0, \min \set{p^-, q^-})$. We express the norm as
  \begin{align*}
    \bignorm{ \set{s_Q}_Q}_{\fspq{\alpha(\cdot)}{p(\cdot)}{q(\cdot)}}
    &= \Biggnorm{ \biggnorm{ 2^{\nu\alpha(x)r} \sum_{Q \in
          \mathcal{D}_\nu} \abs{s_Q}^r\, \abs{Q}^{-\frac{r}{2}}\,
        \chi_Q }_{\lqrxnu} }_{\Lprdotx}^{\frac{1}{r}},
  \end{align*}
  since the sum has only one non-zero term.  We use the estimate
  $\chi_{Q} \leq c\, \eta_{\nu,m} \ast \chi_{E_Q}$ for all $Q \in
  \mathcal{D}_\nu$. Now Lemma~\ref{lem:beta_eta} implies that
  \begin{align*}
    \bignorm{ \set{s_Q}_Q}_{\fspq{\alpha(\cdot)}{p(\cdot)}{q(\cdot)}}
    &\leq c\,\Biggnorm{ \biggnorm{ 2^{\nu\alpha(x)r} \sum_{Q \in
          \mathcal{D}_\nu} \abs{s_Q}^r\, \abs{Q}^{-\frac{r}{2}}\,
        \eta_\nu \ast \chi_{E_Q} }_{\lqrxnu}
    }_{\Lprdotx}^{\frac{1}{r}}
    \\
    &\leq c\,\Biggnorm{ \biggnorm{ \eta_\nu \ast \Big(
        2^{\nu\alpha(\cdot)r} \sum_{Q \in \mathcal{D}_\nu}
        \abs{s_Q}^r\, \abs{Q}^{-\frac{r}{2}}\, \chi_{E_Q} \Big)
      }_{\lqrxnu} }_{\Lprdotx}^{\frac{1}{r}}.
  \end{align*}
  Then Theorem~\ref{thm:eta} completes the proof of the first
  direction:
  \begin{align*}
    \bignorm{
      \set{s_Q}_Q}_{\fspq{\alpha(\cdot)}{p(\cdot)}{q(\cdot)}} &\leq
    c\,\Biggnorm{ \biggnorm{ 2^{\nu\alpha(x)r} \sum_{Q \in
          \mathcal{D}_\nu} \abs{s_Q}^r\, \abs{Q}^{-\frac{r}{2}}\,
        \chi_{E_Q} }_{\lqrxnu} }_{\Lprdotx}^{\frac{1}{r}}
    \\
    &= c\,\Biggnorm{ \biggnorm{ 2^{\nu\alpha(x)} \sum_{Q \in
          \mathcal{D}_\nu} \abs{s_Q}\, \abs{Q}^{-\frac{1}{2}}\,
        \chi_{E_Q} }_{\lqxnu} }_{\Lpdotx}^{\frac{1}{r}} .
  \end{align*}
  The other direction follows by the same argument,
since $\chi_{E_Q} \leq c\, \eta_\nu \ast \chi_{Q}$.
\end{proof}

Next we use the embedding proposition from the previous
section to show that the trace space
does not really depend on the secondary index of integration.

\begin{lemma}\label{lem:trace}
  Let $p_1$, $p_2$, $q_1$, $\alpha_1$ and $\alpha_2$ be as in
  the Standing Assumptions and let $q_2 \in (0, \infty)$.  Assume that
  $\alpha_1=\alpha_2$ and $p_1=p_2$ in the
upper or lower half space, and that $\alpha_1\geq \alpha_2$ and
  $p_1\leq p_2$.  Then
  \begin{align*}
    \trace \Fspq{\alpha_1(\cdot)}{p_1(\cdot)}{q_1(\cdot)}(\R^n) = \trace
    \Fspq{\alpha_2(\cdot)}{p_2(\cdot)}{q_2}(\R^n).
  \end{align*}
\end{lemma}

\begin{proof}
We assume without loss of generality
that $\alpha_1=\alpha_2$ and $p_1=p_2$ in the
upper half space.
  We define $r_0 = \min\set{q_2, q_1^-}$ and $r_1 = \max\set{q_2, q_1^+}$.
    It follows from Proposition~\ref{pro:embedding} that
  \begin{align*}
  \trace
  \Fspq{\alpha_2(\cdot)}{p_2(\cdot)}{r_0} \hookrightarrow \trace
  \Fspq{\alpha_2(\cdot)}{p_1(\cdot)}{q_1(\cdot)} \hookrightarrow \trace
  \Fspq{\alpha_1(\cdot)}{p_1(\cdot)}{r_1}
  \end{align*}
  and
  \begin{align*}
  \trace
  \Fspq{\alpha_2(\cdot)}{p_2(\cdot)}{r_0} \hookrightarrow \trace
  \Fspq{\alpha_2(\cdot)}{p_2(\cdot)}{q_2} \hookrightarrow \trace
  \Fspq{\alpha_1(\cdot)}{p_1(\cdot)}{r_1}.
  \end{align*}

  We complete the proof by showing that $\trace
  \Fspq{\alpha_1(\cdot)}{p_1(\cdot)}{r_1} \hookrightarrow \trace
  \Fspq{\alpha_2(\cdot)}{p_2(\cdot)}{r_0}$. Let $f \in \trace
  \Fspq{\alpha_1(\cdot)}{p_1(\cdot)}{r_1}$. According to
  Theorem~\ref{thm:atom_Fspxq} we have the representation
  \begin{align*}
  f = \sum_{Q\in \mathcal{D}^+} t_Q\, a_Q \quad \text{with} \quad
  \bignorm{ \set{t_Q}_Q}_{\fspq{\alpha_1(\cdot)}{p_1(\cdot)}{r_1}} \leq c
  \norm{f}_{\Fspq{\alpha_1(\cdot)}{p_1(\cdot)}{r_1}},
  \end{align*}
  where the $a_Q$ are smooth atoms for
  $\Fspq{\alpha_1(\cdot)}{p_1(\cdot)}{r_1}$ satisfying \ref{itm:mol1} and
  \ref{itm:mol2} up to high order. Then they are also smooth atoms for
  $\Fspq{\alpha_2(\cdot)}{p_2(\cdot)}{r_0}$.

  Let $A:= \set{Q \in \mathcal{D}^+\,:\, 3\overline{Q} \cap \set{x_n = 0}
    \not=\emptyset}$.  If $Q\in A$ is contained in the closed upper half
  space, then we write $Q\in A^+$, otherwise $Q\in A^-$. We set
  $\widetilde{t}_Q = t_Q$ when $Q\in A$, and $\widetilde{t}_Q = 0$
otherwise.  Then we
  define $\widetilde{f} =\sum_{Q\in \mathcal{D}^+} \widetilde{t}_Q
  \widetilde{a}_Q $.  It is clear that $\trace f = \trace
  \widetilde{f}$, since all the atoms of $f$ whose support intersects
$\R^{n-1}$ are included in $\widetilde{f}$. For $Q\in A^+$ we define
\begin{align*}
E_Q = \bigset{x\in
    Q\,\colon\, \tfrac34 \ell(Q)\le x_n \le \ell(Q)};
\end{align*}
for $Q\in A^-$ we define
\begin{align*}
E_Q = \bigset{(x',x_n)\in \Rn\,\colon\,
(x',-x_n)\in Q,\ \tfrac12 \ell(Q)\le x_n \le \tfrac34 \ell(Q)};
\end{align*}
for all other cubes $E_Q=\emptyset$.
If $Q\in A$, then $\abs{Q}=4\abs{E_Q}$; moreover,
  $\set{E_Q}_Q$ covers each point at most three times.

  By Theorem~\ref{thm:Sphi_inv} and Lemma~\ref{lem:E_Q} we conclude that
  \begin{align*}
    \norm{\widetilde{f}}_{\Fspq{\alpha_2(\cdot)}{p_2(\cdot)}{r_0}}
    &\leq c \bignorm{
      \set{\widetilde{t}_Q}_Q}_{\fspq{\alpha_2(\cdot)}{p_2(\cdot)}{r_0}}
    \leq c \Biggnorm{ \biggnorm{ 2^{\nu\alpha_2(x)} \sum_{Q \in
          \mathcal{D}_\nu} \abs{t_Q}\, \abs{Q}^{-\frac{1}{2}}\,
        \chi_{E_Q} }_{l^{r_0}_\nu} }_{L^{p_2(\cdot)}_x}.
  \end{align*}
The inner norm consists of at most three non-zero members for each $x\in \Rn$.
Therefore, we can replace $r_0$ by $r_1$. Moreover,
  each $E_Q$ is supported in the upper half space, where $\alpha_2$ and $\alpha_1$,
  and $p_2$ and $p_1$ agree. Thus,
  \begin{align*}
    \norm{\widetilde{f}}_{\Fspq{\alpha_2(\cdot)}{p_2(\cdot)}{r_0}}
    \leq c \Biggnorm{
      \biggnorm{ 2^{\nu\alpha_1(x)} \sum_{Q \in \mathcal{D}_\nu}
        \abs{t_Q}\, \abs{Q}^{-\frac{1}{2}}\, \chi_{E_Q} }_{l^{r_1}_\nu}
    }_{L^{p_1(\cdot)}_x}.
  \end{align*}
  The right hand side is bounded by
  $\norm{f}_{\Fspq{\alpha_1(\cdot)}{p_1(\cdot)}{r_1}}$ according to
  Theorem~\ref{thm:Sphi_inv} and Lemma~\ref{lem:E_Q}.  Therefore,
  $\trace\Fspq{\alpha_1(\cdot)}{p_1(\cdot)}{r_1} \hookrightarrow \trace
  \Fspq{\alpha_2(\cdot)}{p_2(\cdot)}{r_0}$, and the claim follows.
\end{proof}


For the next proposition we recall the common notation
$\Fsp{\alpha(\cdot)}{p(\cdot)} = \Fspq{\alpha(\cdot)}{p(\cdot)}{p(\cdot)}$
for the Triebel--Lizorkin space with identical primary and secondary
indices of integrability. The next result shows that the trace space
depends only on the values of the indices at the boundary, as should
be expected.

\begin{proposition}
  \label{prop:trace}
  Let $p_1$, $p_2$, $q_1$, $\alpha_1$ and $\alpha_2$ be as in
  the Standing Assumptions.
  Assume that $\alpha_1(x)=\alpha_2(x)$ and $p_1(x)=p_2(x)$ for all
  $x\in \R^{n-1}\times \set{0}$. Then
  \begin{align*}
    \trace \Fspq{\alpha_1(\cdot)}{p_1(\cdot)}{q_1(\cdot)}(\R^n) =
    \trace \Fsp{\alpha_2(\cdot)}{p_2(\cdot)}(\R^n).
  \end{align*}
\end{proposition}

\begin{proof}
  By Lemma~\ref{lem:trace} we conclude that $\trace
  \Fspq{\alpha_1(\cdot)}{p_1(\cdot)}{q_1(\cdot)} =\trace
  \Fsp{\alpha_1(\cdot)}{p_1(\cdot)}$.  Therefore, we can assume that
  $q_1=p_1$.

  We define $\widetilde{\alpha}_j$ to equal $\alpha_j$ on the lower
  half space and $\min\set{\alpha_1,\alpha_2}$ on the upper half
  space and let $\widetilde{\alpha} = \min\{\alpha_1,\alpha_2\}$.
Similarly, we define $\widetilde{p}_j$ and $\widetilde p$.
Applying Lemma~\ref{lem:trace} four times in the following chain
  \begin{align*}
    \trace \Fsp{\alpha_1(\cdot)}{p_1(\cdot)}(\R^n) = \trace
    \Fsp{\widetilde{\alpha}_1(\cdot)}{\widetilde{p}_1(\cdot)}(\R^n)
= \trace
    \Fsp{\widetilde{\alpha}(\cdot)}{\widetilde{p}(\cdot)}(\R^n)
    = \trace
    \Fsp{\widetilde{\alpha}_2(\cdot)}{\widetilde{p}_2(\cdot)}(\R^n)
    = \trace \Fsp{\alpha_2(\cdot)}{p_2(\cdot)}(\R^n),
  \end{align*}
gives the result.
\end{proof}

\begin{proof}[Proof of Theorem~\ref{thm:trace}]
  By Proposition~\ref{prop:trace} it suffices to consider the case
  $q=p$ with $p$ and $\alpha$ independent of the $n$-th coordinate for
  $\abs{x_n}\le 2$.
   Let $f\in \Fsp{\alpha(\cdot)}{p(\cdot)}$ with $\|
  f\|_{\Fsp{\alpha(\cdot)}{p(\cdot)}} \leq 1$ and let $f=\sum s_Q a_Q$
  be an atomic decomposition as in Theorem~\ref{thm:atom_Fspxq}.

  We denote by $\pi$ the orthogonal projection of $\Rn$ onto
  $\R^{n-1}$, and $(x',x_n)\in \Rn = \R^{n-1} \times \R$.  For $J\in
  \mathcal{D}^{n-1}_\mu$, a dyadic cube in $\R^{n-1}$, we define $Q_i(J)
  \in \mathcal{D}_\mu^{n}$, $i=1,\ldots, 6\cdot 5^{n-1}$,
to be all the dyadic cubes satisfying $J\subset 3 Q_i$. We define $t_J =
  \abs{Q_1(J)}^{-\frac1{2n}} \sum_i \abs{s_{Q_i(J)}}$ and
$h_J(x')= t_J^{-1} \sum_i s_{Q_i} a_{Q_i}$. By $Q_+(J)$ we denote the
cube $Q_i(J)$ which has $J$ as a face (i.e.\ $J\subset \partial Q_+(J)$).

  Then we have
  \begin{align*}
  \trace f(x') = \sum_\mu \sum_{J\in \mathcal{D}^{n-1}_\mu} t_J h_J(x'),
  \end{align*}
  with convergence in $\mathcal{S}'$. The condition $\alpha-
  \frac{1}{p} - (n-1) \Big(\frac1p - 1\Big)_+>0$ implies that
molecules in $\Fsp{\alpha(\cdot)-
    \frac{1}{p(\cdot)}}{p(\cdot)}(\R^{n-1})$ are not required to
satisfy any moment conditions. Therefore, $h_J$ is a
  family of smooth molecules for this space. Consequently, by
Theorem~\ref{thm:Sphi_inv}, we find that
  \begin{align*}
  \| \trace f\|_{\Fsp{\alpha(\cdot)- \frac{1}{p(\cdot)}}{p(\cdot)}(\R^{n-1})} \leq
  c \| \set{t_J}_J \|_{\fsp{\alpha(\cdot)- \frac{1}{p(\cdot)}}{p(\cdot)}(\R^{n-1})}.
  \end{align*}
  Thus, we conclude the proof by showing that the right hand side is
  bounded by a constant. Since the norm is bounded if and only if the modular
is bounded, we see that it suffices to show that
  \begin{align*}
 & \int_{\R^{n-1}} \sum_\mu \sum_{J\in \mathcal{D}^{n-1}_\mu} \Big(
  2^{\mu \big(\alpha(x',0)-\frac1{p(x',0)} \big)} \abs{t_J} |J|^{-1/2}
  \chi_J(x',0)\Big)^{p(x',0)} dx' \\
 &\qquad =   \sum_\mu \sum_{J\in \mathcal{D}^{n-1}_\mu} 2^{-\mu}\int_J \Big( 2^{\mu
    \alpha(x',0)} \abs{t_J} |J|^{-1/2}
  \Big)^{p(x',0)} dx'
  \end{align*}
is bounded. For the integral we calculate
  \begin{align*}
    2^{-\mu}\int_J  \Big( 2^{\mu \alpha(x',0)} \abs{t_J}
    |J|^{-1/2} \Big)^{p(x',0)} dx'
    & =  \int_{Q_+(J)} \Big( 2^{\mu \alpha(x',0)}
    \abs{t_J} |J|^{-1/2} \Big)^{p(x',0)} d(x',x_n)
    \\
    & \le c  \int_{Q_+(J)} \Big( 2^{\mu \alpha(x)}
    \sum_i\abs{s_{Q_i}} |Q|^{-\frac1{2n}-\frac{n-1}{2n} }
    \Big)^{p(x)} dx
    \\
    & = c \int_{Q_+(J)} \Big( 2^{\mu \alpha(x)} \sum_i\abs{s_{Q_i}} |Q|^{-1/2} \Big)^{p(x)} dx
  \end{align*}

  Hence, we obtain
  \begin{align*}
    \sum_\mu \sum_{J\in \mathcal{D}^{n-1}_\mu}  & 2^{-\mu}\int_J \Big( 2^{\mu
      \alpha(x',0)}
    \abs{t_J} |J|^{-1/2} \Big)^{p(x',0)} dx' \\
     &\le c\, \sum_\mu \sum_{Q\in \mathcal{D}^{n}_\mu} \int_Q \Big( 2^{\mu
      \alpha(x)}  \sum_i\abs{s_{Q_i}}|Q|^{-1/2}
    \Big)^{p(x)} dx \\
 &\le c \int_{\Rn} \sum_\nu \sum_{Q\in \mathcal{D}^n_\nu} \Big( 2^{\nu
    \alpha(x)} \abs{s_Q} |Q|^{-1/2} \chi_Q(x)\Big)^{p(x)} dx,
  \end{align*}
where we again swapped the integral and the sums. Since
$\| f\|_{\Fsp{\alpha(\cdot)}{p(\cdot)}} \leq 1$,
the right hand side quantity is bounded, and we are done.
\end{proof}


\appendix


\section{Technical lemmas}
\label{sect:technical}

Recall from \eqref{E:nu} 
that
$\eta_{\nu, m}(x) = 2^{n \nu}(1+2^\nu |x|)^{-m}$.

\begin{lemma}
  \label{lem:eta_nu_monoton}
  Let $\nu_1 \geq \nu_0$, $m > n$, and $y \in \R^n$. Then
  \begin{alignat*}{2}
    \eta_{\nu_0,m}(y) &\leq 2^m\,\eta_{\nu_1,m}(y) &\qquad &\text{if }
    \abs{y} \leq 2^{-\nu_1}; \text{ and}
    \\
    \eta_{\nu_1,m}(y) &\leq 2^m\,\eta_{\nu_0,m}(y) &\qquad &\text{if }
    \abs{y} \geq 2^{-\nu_0}.
  \end{alignat*}
\end{lemma}

\begin{proof}
  Let $\abs{y} \leq
  2^{-\nu_1}$. Then $1+ 2^{\nu_1} \abs{y} \leq 2$ and
  \begin{align*}
    \frac{\eta_{\nu_0,m}(y)}{\eta_{\nu_1,m}(y)} &= \frac{2^{n\nu_0}
      (1+ 2^{\nu_1} \abs{y})^{m}}{2^{n\nu_1} (1+ 2^{\nu_0}
      \abs{y})^{m}} \leq \frac{2^{n\nu_0} \cdot 2^m}{2^{n\nu_1}} \leq
    2^m,
  \end{align*}
  which proves the first inequality. Assume now that $\abs{y} \geq
  2^{-\nu_0}$. Then $1+2^{\nu_0} \abs{y} \leq 2 \cdot 2^{\nu_0}
  \abs{y}$ and
  \begin{align*}
    \frac{\eta_{\nu_1,m}(y)}{\eta_{\nu_0,m}(y)} &= \frac{2^{n\nu_1}
      (1+ 2^{\nu_0} \abs{y})^{m}}{2^{n\nu_0} (1+ 2^{\nu_1}
      \abs{y})^{m}} \leq \frac{2^{n\nu_1} (2 \cdot 2^{\nu_0}
      \abs{y})^{m}}{2^{n\nu_0} (2^{\nu_1} \abs{y})^{m}} = 2^m\,
    2^{(\nu_1 - \nu_0)(n-m)} \leq 2^m,
  \end{align*}
  which gives the second inequality.
\end{proof}

\begin{lemma}
  \label{lem:conv_eta_chi}
  Let $\nu \geq 0$ and $m>n$. Then for $Q \in \mathcal{D}_\nu$,
  $y \in Q$ and $x \in \R^n$, we have
  \begin{align*}
 \eta_{\nu,m} \ast
    \bigg(\frac{\chi_Q}{\abs{Q}} \bigg)(x) \approx \eta_{\nu,m} (x-y).
  \end{align*}
\end{lemma}

\begin{proof} Fix $Q \in \mathcal{D}_\nu$ and set $d=1+\sqrt{n}$.
  If $y, z \in Q$, then $\abs{y-z} \leq \sqrt{n}\, 2^{-\nu}$ and
  \begin{align*}
    \tfrac{1}{d}\, (1+2^{\nu} \abs{x-z})&\le
    1+ \tfrac{1}{d} \cdot 2^{\nu}
    (\abs{x-z} - \sqrt{n}\, 2^{-\nu})  \\
    &\le 1+2^{\nu} \abs{x-y} \\
    &\leq 1+2^{\nu}\, (\abs{x-z} +
    \sqrt{n} \, 2^{-\nu})  \leq d\, (1+2^{\nu} \abs{x-z}).
  \end{align*}
  Therefore, for all $y, z \in Q$ we have
  \begin{align*}
    2^{-m} \eta_{\nu,m}(x-y) \leq \eta_{\nu,m}(x-z)  \leq 2^m
    \eta_{\nu,m}(x-y).
  \end{align*}
The claim follows when we integrate this estimate over $z\in Q$
and use the formula
  \[
     \eta_{\nu, m} \ast \bigg(\frac{\chi_{Q}}{\abs{Q}} \bigg)(x) =
    \frac{1}{\abs{Q}} \int_{Q} \eta_{\nu, m} (x-z)\,dz. \qedhere
  \]
\end{proof}

\begin{lemma}
  \label{lem:conv_eta_eta}
  For $\nu_0, \nu_1 \geq 0$ and $m > n$, we have
  \begin{align*}
    \eta_{\nu_0,m} \ast \eta_{\nu_1,m} &\approx
    \eta_{\min\set{\nu_0,\nu_1}, m}
  \end{align*}
  with the constant depending only on $m$ and $n$.
\end{lemma}

\begin{proof}
  Using dilations and symmetry
we may assume that $\nu_0 = 0$ and $\nu_1 \geq 0$.
  Since $m > n$, we have $\norm{\eta_{\nu_0,m}}_1 \leq c$ and
  $\norm{\eta_{\nu_1,m}}_1 \leq c$.

  We start with the direction ``$\ge$''.  If $\abs{y} \leq
  2^{-\nu_1} \leq 1$, then $1+\abs{x-y} \leq 2(1+\abs{x})$, and
  therefore, $\eta_{\nu_0,m}(x-y) \geq c\, \eta_{\nu_0,m}(x)$.  Hence,
  \begin{align*}
    \eta_{\nu_0,m} \ast \eta_{\nu_1,m}(x) &\geq \int_{\set{y\,:\,
        \abs{y} \leq 2^{-\nu_1}}} \eta_{\nu_0,m}(x-y)\,
    \eta_{\nu_1,m}(y)\,dy
    \\
    &\geq c\,\eta_{\nu_0,m}(x)\,\int_{\set{y\,:\, \abs{y} \leq
        2^{-\nu_1}}} 2^{n\nu_1} (1+2^{\nu_1} \abs{y})^{-m}\,dy
    \\
    &\geq c\,\eta_{\nu_0,m}(x)\,\int_{\set{y\,:\, \abs{y} \leq
        2^{-\nu_1}}} 2^{n\nu_1}\, 2^{-m}\,dy
    \\
    &\geq c\,2^{-m}\,\eta_{\nu_0,m}(x).
  \end{align*}

  We now prove the opposite direction, ``$\le$''.  Let
  $A := \set{y \in \R^n\,:\, \abs{y} \leq 3\ \text{or}\ \abs{x-y} > \abs{x}/2}$.
  If  $y \in A$, then $1+\abs{x-y} \geq
  \frac{1}{4}(1+\abs{x})$, which implies that $\eta_{\nu_0,m}(x-y) \leq c\,
  \eta_{\nu_0,m}(x)$ and
  \begin{align*}
    \int_{A} \eta_{\nu_0,m}(x-y)\, \eta_{\nu_1,m}(y)\,dy &\leq
    c\, \eta_{\nu_0,m}(x) \int_{A} \eta_{\nu_1,m}(y)\,dy\, \leq
    c\, \eta_{\nu_0,m}(x).
  \end{align*}
  If $y\in \Rn\setminus A$, then $\abs{y} \geq 1$ and $\abs{y} \geq \frac{1}{2}
  \abs{x}$. So $\eta_{\nu_1,m}(y) \leq c\,\eta_{\nu_0,m}(y) \leq
    c\,\eta_{\nu_0,m}(x)$ by Lemma~\ref{lem:eta_nu_monoton}. Hence,
  \begin{align*}
    \int_{\Rn\setminus A} \eta_{\nu_0,m}(x-y)\, \eta_{\nu_1,m}(y)\,dy
&\leq c\,\int_{\Rn\setminus A}
    \eta_{\nu_0,m}(x-y)\, dy\, \eta_{\nu_0,m}(x) \leq
    c\,\eta_{\nu_0,m}(x).
  \end{align*}
  Combining the estimates over $A$ and $\Rn \setminus A$ gives
  \[
    \eta_{\nu_0,m} \ast \eta_{\nu_1,m}(x)
\leq c\,
    \eta_{\min\set{\nu_0,\nu_1}, m}(x). \qedhere
  \]
\end{proof}

\begin{lemma}
  \label{lem:etapower}
  Let $r\in(0,1]$. Then for $\nu$, $\mu \geq 0$, $m > \frac{n}{r}$
  and $Q_\mu \in \mathcal{D}_\mu$, we have
  \begin{align*}
    \big( \eta_{\nu,m} \ast \eta_{\mu,m}\ast \chi_{Q_\mu} \big)^r
    &\approx 2^{(\mu-\nu)_+ n(1-r)}\,\eta_{\nu,mr} \ast
    \eta_{\mu,mr}\ast \chi_{Q_\mu},
  \end{align*}
  where the constant depends only on $m$, $n$ and $r$.
\end{lemma}

\begin{proof}
  Without loss of generality, we may assume that $x_{Q_\mu} = 0$.
  Then by Lemmas ~\ref{lem:conv_eta_chi} and \ref{lem:conv_eta_eta}
  \begin{alignat*}{2}
    \eta_{\nu,m} \ast \eta_{\mu,m}\ast \chi_{Q_\mu} &\approx
    2^{-n\mu}\,\eta_{\nu,m} \ast \eta_{\mu,m} &&\approx
    2^{-n\mu}\,\eta_{\min\{\nu,\mu\},m},
    \\
    \eta_{\nu,mr} \ast \eta_{\mu,mr}\ast \chi_{Q_\mu} &\approx
    2^{-n\mu}\,\eta_{\nu,mr} \ast \eta_{\mu,mr} &&\approx
    2^{-n\mu}\,\eta_{\min\{\nu,\mu\},mr}.
  \end{alignat*}
{}From the definition of $\eta$ we get
\begin{align*}
\big( \eta_{\min\set{\nu,\mu},m}\big)^r = 2^{\min\set{\nu,\mu}n
(r-1)}\, \eta_{\min\set{\nu,\mu},mr}.
\end{align*}
Thus, we get
  \[
    \big( \eta_{\nu,m} \ast \eta_{\mu,m}\ast \chi_{Q_\mu} \big)^r
    \approx 2^{\mu n(1-r)} 2^{\min\set{\nu,\mu}n(r-1)}\, \eta_{\nu,mr}
    \ast \eta_{\mu,mr}\ast \chi_{Q_\mu}. \qedhere
  \]
\end{proof}

\begin{lemma}
  \label{lem:van_moments}
  Let $g,h \in L^1_\loc(\R^n)$ and $k \in \N_0$ such that
  $D^\mu g \in L^1(\R^n)$ for all multi-indices $\mu$ with $|\mu|\le k$.
Assume that there exist $m_0 >n$ and
  $m_1 > n+k$ such that $\abs{h} \leq \eta_{\mu,m_1}$ and
  $\abs{D^\mu g}\leq 2^{\nu k}\,\eta_{\nu,m_0}$. Further, suppose that
  \begin{align*}
    \int_{\Rn} x^{\gamma} h(x) \, dx = 0, \quad\text{for } ~~\abs{\gamma} \leq
    k-1.
  \end{align*}
  Then \quad $\displaystyle \bigabs{g \ast h} \leq c
  \, 2^{k(\nu-\mu)} \,\eta_{\nu,m_0} \ast \eta_{\mu,m_1-k}$.
\end{lemma}

\begin{proof}
  If $k=0$, then the estimate is obvious, so we can assume $k \geq 1$.
  It suffices to prove the result for $g,h$ smooth.  Since $h$ has
  vanishing moments up to order $k-1$ we estimate by Taylor's formula
  \begin{align*}
    \bigabs{g \ast h(x)}
&\leq \int_{\R^n} \biggabs{ \bigg( g(y)
      - \sum_{\abs{\gamma}\leq k-1} D^\gamma g(x)\,
      \frac{(y-x)^\gamma}{\gamma!} \bigg)\, h(x-y)} \,dy
    \\
    &\leq c\,\int_{\R^n} \int_{[x,y]} \sup_{|\mu|=k}|D^\mu g(\xi)|\,
    \abs{x-\xi}^{k-1}\,d\xi\, \abs{h(x-y)} \,dy
    \\
    &\leq c\,\int_{\R^n} \int_{[x,y]} 2^{\nu k}\,
    \eta_{\nu,m_0}(\xi)\, \abs{x-\xi}^{k-1} \eta_{\mu,m_1}(x-y) \,d\xi \,dy
  \end{align*}
Changing the order of integration with $y - x = r(\xi - x)$, where $r\ge 1$,
yields the inequality
  \begin{align}
   \label{E:star3}
   \bigabs{g \ast h(x)}
  & \le c\,2^{\nu k}\, \int_{\R^n} \int_1^\infty
    \eta_{\nu,m_0}(\xi)\, \abs{x-\xi}^{k}\,
    \eta_{\mu,m_1} \big(r(x-\xi)\big) \,dr \,d\xi.
  \end{align}
We estimate the inner integral: for $2^\mu r |x-\xi|\ge 1$ we have
$\eta_{\mu,m_1} \big(r(x-\xi)\big) \approx r^{-m_1}
\eta_{\mu,m_1}(x-\xi)$; for $2^\mu r \, |x-\xi|<1$ we simply use
$\eta_{\mu,m_1} \big(r(x-\xi)\big) \le \eta_{\mu,m_1}(x-\xi)$. Thus,
we find that
\begin{align*}
\begin{split}
\int_1^\infty \eta_{\mu,m_1} \big(r(x-\xi)\big) \,dr &\le
\bigg(\int_1^\infty r^{-1-m_1} \, dr
 + \int_1^{2^{-\mu} |x-\xi|^{-1}} r^{-1}\, dr \bigg) \,
 \eta_{\mu,m_1}(x-\xi) \\
&\approx \log \big(e+2^{-\mu} |x-\xi|^{-1}\big) \eta_{\mu,m_1}(x-\xi).
\end{split}
\end{align*}
Substituting this into \eqref{E:star3} produces
\begin{align*}
    \bigabs{g \ast h(x)} &\leq c\, 2^{\nu k}\, \int_{\R^n}
\eta_{\nu,m_0}(\xi)\, \abs{x-\xi}^k  \log \big(e+2^{-\mu}
|x-\xi|^{-1}\big) \, \eta_{\mu,m_1}(x-\xi)
    \,d\xi
    \\
    & = c\,2^{\nu k}\, \int_{\R^n}
\log \big(e+2^{-\mu} |x-\xi|^{-1}\big)
\bigg(\frac{\abs{x-\xi}}{1+2^\mu\abs{x-\xi}}\bigg)^{k}
\eta_{\nu,m_0}(\xi)\, \eta_{\mu,m_1-k}(x-\xi) \,d\xi
    \\
    &\le 2^{k(\nu-\mu)}\, \eta_{\nu,m_0} \ast \eta_{\mu,m_1-k}(x),
  \end{align*}
  proving the assertion.
\end{proof}

\begin{lemma}[``The $r$-trick"]\label{lem:est_g}
  Let $r>0$, $\nu \geq 0$ and $m > n$. Let $x \in \R^n$.
  Then there exists $c =c(r,m,n) > 0$ such that for all
  $g \in \mathcal{S}'$ with $\supp \hat{g}
  \subset \set{ \xi \,:\, \abs{\xi} \leq 2^{\nu+1}}$, we have
\begin{align*}
\abs{g(x)}
   \leq c\, \big(\eta_{\nu,m} \ast
  \abs{g}^r(x)\big)^{{1}/{r}}.
\end{align*}
\end{lemma}

\begin{proof}
Fix a dyadic cube $Q=Q_{\nu,k}$ and $x\in Q$.
By (2.11) of \cite{FJ1} we have
\begin{align*}
g(x) \le \sup_{z \in Q} |g(z)|^r \leq c_r \, 2^{\nu n} \sum_{l \in
\Z^n} (1+ \abs{l})^{-m} \int_{Q_{\nu, k+l}} \abs{g(y)}^r\,dy.
\end{align*}
In the reference this was shown only for $m=n+1$, but it is easy to see that it
is also true for $m>n+1$. Now for $x \in Q_{\nu, k}$ and $y \in
Q_{\nu, k+l}$, we have $|x-y| \approx 2^{-\nu} |l|$ for large $l$,
hence, $1+2^\nu |x-y| \approx 1+ |l|$.
{}From this we conclude that
\begin{align*}
\sup_{z \in Q} |g(z)|^r  &\leq c_{r,n} \, \sum_{l \in \Z^n}
\int_{Q_{\nu, k+l}} (1+ 2^{\nu}\abs{x-y})^{-m} \, \abs{g(y)}^r \,dy\\
& = c_{r, n} \, \int_{\R^n} 2^{\nu n} (1+ 2^\nu \abs{x-z})^{-m}
\abs{g(z)}^r\,dz = c_{r, n}\, \eta_{\nu,m} \ast \abs{g}^r(x).
\end{align*}
Now, taking the $r$-th root, we obtain the claim.
\end{proof}


\subsection*{Acknowledgment}

We would like to thank H.-G.~Leopold for useful discussions on how our
spaces relate to his spaces of variable smoothness, and
Tuomas Hyt{\"o}nen for a piece of advice on Fourier analysis.

The first author thanks Arizona State University
and the University of Oulu for their hospitality. All authors
thank the W.\ Pauli Institute, Vienna, at which it was possible to
complete the project.



\end{document}